\documentclass[12pt]{amsart}

\usepackage[width=17.00cm, height=20.00cm]{geometry}
\usepackage{amsfonts}
\usepackage{graphicx}
\usepackage{amssymb}
\usepackage{amsmath}
\usepackage{hyperref}
\usepackage{enumerate}

\newtheorem{corollary}{Corollary}
\newtheorem{definition}{Definition}
\newtheorem{example}{Example}
\newtheorem{lemma}{Lemma}
\newtheorem{notation}{Notation}

\newtheorem{proposition}{Proposition}
\newtheorem{remark}{Remark}

\newtheorem{theorem}{Theorem}
\numberwithin{equation}{section}

\begin{document}
\title[Asymptotic almost automorphy]{Asymptotic almost automorphy \\
for algebras of generalized functions}
\author{Chikh BOUZAR}
\address{ Laboratory of Mathematical Analysis and Applications. Universit%
\'{e} Oran 1, Ahmed Ben Bella, 31000, Oran, Algeria.}
\email{ch.bouzar@gmail.com}
\author{Meryem SLIMANI}
\address{Laboratory of Mathematical Analysis and Applications. Universit\'{e}
Oran 1, Ahmed Ben Bella, 31000, Oran, Algeria.}
\email{meryemslimani@yahoo.com}

\begin{abstract}
Asymptotic almost automorphy is introduced and studied in the context of
some algebras of generalized functions. We also give applications to neutral
difference differential systems in the framework of such generalized
functions.
\end{abstract}

\date{}
\maketitle

\section{Introduction}

Bochner S. defined explicitly almost automorphic functions in the papers
\cite{Bochner61}-\cite{Bochner64}, where also some of their basic properties
are given. In \cite{Bochner62} he studied linear difference differential
equations in the framework of such functions. It is well known that the
concept of almost automorphy is strictly more general than the almost
periodicity of H.\ Bohr \cite{Bohr}, however the Stepanoff almost
periodicity \cite{Step} and the Levitan almost periodicity \cite{Lev} don't
enter into the Bochner concept. Asymptotic almost periodicity of functions
as a perturbation of almost periodic functions by functions vanishing at
infinity is due to Fr\'{e}chet M. in \cite{Frechet}. Asymptotic almost
automorphy of classical functions is considered in \cite{Ngu}, see also \cite%
{BT-AAAD}. The almost periodicity and the asymptotic almost periodicity of
Sobolev-Schwartz distributions, \cite{Sob} and \cite{Schw}, are respectively
considered by L.\ Schwartz in \cite{Schw} and I.\ Cioranescu in \cite{Cior},
while almost automorphy and asymptotic almost automorphy in the setting of\
these distributions are respectively the subject of the recent works \cite%
{BT-AAD}\ and \cite{BT-AAAD}.

In view of the result \cite{Schw2} on the impossibility\ of the
multiplication of distributions, see \cite{Oberg Book} for more details,
algebras of generalized functions containing spaces of Sobolev-Schwartz type
distributions have been studied, see \cite{Colb}, \cite{Egor}, \cite{AntRad}
and \cite{NPS}. The concepts of almost periodicity and asymptotic almost
periodicity as well as almost automorphy in the context of such algebras of
generalized functions are introduced, studied and applied in the papers \cite%
{BK-APGF}, \cite{BK-APGFDE} \cite{BK-AAPGF}, \cite{BK-AAPGFDE} and \cite%
{BKT-AAGF}. So, the paper first introduces and studies a class of
asymptotically almost automorphic generalized functions, denoted by $%
\mathcal{G}_{aaa}.$ In the sense of multiplication, not only $\mathcal{G}%
_{aaa}$\ is stable under multiplication and it contains the space of
asymptotically almost automorphic distributions of \cite{BT-AAAD}, but
moreover some nonlinear operations are performed within the algebra $%
\mathcal{G}_{aaa}.$ As a by pass result, we give a Seeley type result on
extention of functions in the context of the introduced generalized
functions, this is needed in the proof of a fundamental result on the
uniqueness of decomposition of an asymptotically almost automorphic
generalized function. The papers \cite{BK-AAPGF}\ and \cite{BK-AAPGFDE} can
be considered as consequences of this work. The paper aims also, as in \cite%
{Frechet}, to lift a Frechet existence result of asymptotically almost
automorphic solutions of differential equations to the level of neutral
difference differential systems in the framework of $\mathcal{G}_{aaa}.$

It is worth noting that the meaning of generalized functions is utilized
differently by authors as distributions or ultradistributions, even as
hyperfunctions, but in this work by generalized functions we mean in the
sense of the works \cite{Colb}, \cite{Egor}, \cite{AntRad} and \cite{NPS}.

The paper is organized as follows: section two recalls definitions and some
properties of asymptotically almost automorphic functions and asymptotically
almost automorphic distributions as in \cite{BT-AAAD}. Section three
introduces asymptotically almost automorphic generalized functions and gives
some of their important properties. The study of a Seeley type result on
extensions of generalized functions is given in section four. In section six
non-linear operations on asymptotically almost automorphic generalized
function are studied. The last section is dedicated to linear neutral
difference differential systems in the framework of asymptotically almost
automorphic generalized functions.

\section{Asymptotic almost automorphy of functions and distributions}

Let $\mathcal{C}_{b}$ denotes the space of bounded and continuous
complex-valued functions defined on $\mathbb{R},$ endowed with the norm $%
\left\Vert \cdot \right\Vert _{L^{\infty }\left( \mathbb{R} \right) }$ of
uniform convergence on $\mathbb{R},$ it is well-known that $\left( \mathcal{C%
}_{b},\left\Vert \cdot \right\Vert _{L^{\infty }\left( \mathbb{R} \right)
}\right) $ is a Banach algebra.

A complex-valued function $g~$defined and continuous on $\mathbb{R}$ is
called almost automorphic if for any sequence $\left( s_{m}\right) _{m\in
\mathbb{N}}\subset \mathbb{R},$ one can extract a subsequence $\left(
s_{m_{k}}\right) _{k}$ such that%
\begin{equation*}
\tilde{g}\left( x\right) :=\lim_{k\rightarrow +\infty }g\left(
x+s_{m_{k}}\right) ~~\text{exists for every }x\in \mathbb{R},
\end{equation*}%
and%
\begin{equation*}
\lim_{k\rightarrow +\infty }\tilde{g}\left( x-s_{m_{k}}\right) =g\left(
x\right) \text{~for every }x\in \mathbb{R}.
\end{equation*}%
The space of almost automorphic functions on $\mathbb{R}$ is denoted by $%
\mathcal{C}_{aa}.$

The space $\mathcal{C}_{+,0}$ is the set of all bounded and continuous
complex-valued functions defined on $\mathbb{R}$ and vanishing at $+\infty .$

\begin{definition}
We say that a function $f$ $\in \mathcal{C}_{b}\mathcal{~}$ is
asymptotically almost automorphic, if there exist $g\in \mathcal{C}_{aa}$
and $h\in \mathcal{C}_{+,0}$ such that $f=g+h$ on $\mathbb{J}:=\left[
0,+\infty \right[ .$ The space of asymptotically almost automorphic
functions is denoted by $\mathcal{C}_{aaa}.$
\end{definition}

For a study of \ asymptotically almost automorphic functions and
asymptotically almost automorphic distributions see \cite{BT-AAD} and \cite%
{BT-AAAD} and the references list therein.

\begin{proposition}
The decomposition of an asymptotically almost automorphic function is unique
on $\mathbb{J}.$
\end{proposition}

\begin{notation}
If $f\in \mathcal{C}_{aaa}$ and $f=g+h$ on $\mathbb{J},$ where $g\in
\mathcal{C}_{aa}$ and $h\in \mathcal{C}_{+,0}.~$Due to the uniqueness of the
decomposition of $f,$ the function $g$ is said the principal~term of$~f$ and
the function $h$~the corrective term of$~f,$ we denote them respectively by $%
f_{aa}$ and $f_{cor}.$ The notation $f=\left( f_{aa}+f_{cor}\right) \in
\mathcal{C}_{aaa}$ means that $f_{aa}\in \mathcal{C}_{aa},~f_{cor}\in
\mathcal{C}_{+,0}$ and $f=f_{aa}+f_{cor}$ on $\mathbb{J}.$
\end{notation}

Let $\mathcal{E}\left( \mathbb{I}\right) $ be the algebra of space of smooth
functions on $\mathbb{I=\mathbb{R} }$ or $\mathbb{J},$ and define the space
\begin{equation*}
\mathcal{D}_{L^{p}}\left( \mathbb{I}\right) :=\left\{ \varphi \in \mathcal{E}%
(\mathbb{I}):\forall j\in \mathbb{Z}_{+},\varphi ^{(j)}\in L^{p}\left(
\mathbb{I}\right) \right\} ,p\in \left[ 1,+\infty \right] ,
\end{equation*}%
that we endow with the topology defined by the family of semi-norms
\begin{equation*}
\left\vert \varphi \right\vert _{k,p,\mathbb{I}}\;:=\sum\limits_{j\leq
k}\left\Vert \varphi ^{\left( j\right) }\right\Vert _{L^{p}\left( \mathbb{I}%
\right) }\;,k\in \mathbb{Z}_{+}.
\end{equation*}%
So, $\mathcal{D}_{L^{p}}\left( \mathbb{I}\right) $ is a Fr\'{e}chet
subalgebra of $\mathcal{E}\left( \mathbb{I}\right) .$ Denote $\mathcal{B}%
\left( \mathbb{I}\right) :=\mathcal{D}_{L^{\infty }}\left( \mathbb{I}\right)
.$

\begin{remark}
We have $\lim\limits_{x\underset{>}{\rightarrow }0}\varphi ^{(j)}(0)$ exists
for every $j\in \mathbb{Z} _{+}$ when $\varphi \in \mathcal{D}_{L^{p}}\left(
\mathbb{J}\right) .$
\end{remark}

The space of smooth almost automorphic functions $\mathcal{B}_{aa}$\ and
smooth asymptotically almost automorphic functions $\mathcal{B}_{aaa}$ are
defined respectively by%
\begin{equation*}
\mathcal{B}_{aa}:=\left\{ \varphi \in \mathcal{E}\left( \mathbb{R} \right)
:\forall j\in \mathbb{Z}_{+},~\varphi ^{\left( j\right) }\in \mathcal{C}%
_{aa}\right\} .
\end{equation*}%
\begin{equation*}
\mathcal{B}_{aaa}:=\left\{ \varphi \in \mathcal{E}\left( \mathbb{R} \right)
:~\forall j\in \mathbb{Z}_{+},~\varphi ^{\left( j\right) }\in \mathcal{C}%
_{aaa}\right\} .
\end{equation*}%
We endow $\mathcal{B}_{aa}$ and $\mathcal{B}_{aaa}$ with the topology
induced by $\mathcal{B}:=\mathcal{D}_{L^{\infty }}\left( \mathbb{R}\right) .$

\begin{proposition}
\label{prop1.1}

\begin{enumerate}
\item The space $\mathcal{B}_{aaa}~$ is a Fr\'{e}chet subalgebra of $%
\mathcal{B}$ stable by translation.

\item $\mathcal{B}_{aaa}\times \mathcal{B}_{aa}\subset \mathcal{B}_{aaa}.$

\item $\mathcal{B}_{aaa}\ast L^{1}\subset \mathcal{B}_{aaa}.$
\end{enumerate}
\end{proposition}

The space of $L^{p}-$distributions, $p\in \left] 1,+\infty \right] ,$
denoted by $\mathcal{D}_{L^{p}}^{\prime }\left( \mathbb{R} \right) ,~$is the
topological dual of $\mathcal{D}_{L^{q}}\left( \mathbb{R} \right) ,$ where $%
\frac{1}{p}+\frac{1}{q}=1.$ Let $\mathcal{\dot{B}}$\ be the closure in $%
\mathcal{B}~$of the space $\mathcal{D\subset E}\left( \mathbb{R} \right) $
of functions with compact support. The topological dual of $\mathcal{\dot{B}}
$ is denoted by $\mathcal{D}_{L^{1}}^{\prime }\left( \mathbb{R} \right) .$
The space of bounded distributions $\mathcal{D}_{L^{\infty }}^{\prime
}\left( \mathbb{R} \right) $ is denoted by $\mathcal{B}^{\prime }.$

\begin{definition}
The space of almost automorphic distributions, denoted by $\mathcal{B}%
_{aa}^{\prime },$\ is the space of $~T\in \mathcal{B}^{\prime }$ satisfying
one of the following equivalent statements

\begin{enumerate}
\item $T\ast \varphi \in \mathcal{C}_{aa},~\forall \varphi \in \mathcal{D}.$

\item $\exists k\in \mathbb{Z}_{+}$ and $g_{j}\in \mathcal{C}_{aa},$ $0\leq
j\leq k,$ such that $T=\sum\limits_{j=0}^{k}g_{j}^{\left( j\right) }.$
\end{enumerate}
\end{definition}

The space of bounded distributions vanishing at infinity, denoted by $%
\mathcal{B}_{+,0}^{\prime },$ is the space of $~Q\in \mathcal{B}^{\prime }~$
satisfying
\begin{equation*}
\lim\limits_{\omega \rightarrow +\infty }\left\langle \tau _{\omega
}Q,\varphi \right\rangle :=\lim\limits_{\omega \rightarrow +\infty
}\left\langle Q,\tau _{-\omega }\varphi \right\rangle =0,~\forall \varphi
\in \mathcal{D},
\end{equation*}%
where $\tau _{\omega }\varphi \left( \cdot \right) :=\varphi \left( \cdot
+\omega \right) ,\omega \in \mathbb{R} .$

\begin{theorem}
\label{def-AAAD} The space of asymptotically almost automorphic
distributions, denoted by $\mathcal{B}_{aaa}^{\prime },$ is the space of $\
T\in \mathcal{B}^{\prime }~$satisfying one of the following equivalent
statements

\begin{enumerate}
\item $\exists P\in \mathcal{B}_{aa}^{\prime },\exists $ $Q\in \mathcal{B}%
_{+,0}^{\prime }~$such that $T=P+Q$ on $\mathbb{J}.$

\item $T\ast \varphi \in $ $\mathcal{C}_{aaa},~\forall \varphi \in \mathcal{D%
}.$

\item $\exists k\in \mathbb{Z}_{+}$ and $f_{j}\in \mathcal{C}_{aaa},$ $0\leq
j\leq k,$ such that $T=\sum\limits_{j=0}^{k}\ f_{j}^{\left( j\right) }.$

\item $\exists \left( \theta _{m}\right) _{m\in \mathbb{N}}\subset \mathcal{B%
}_{aaa}$ such that $\lim\limits_{m\rightarrow +\infty }\theta _{m}=T$ in $%
\mathcal{B}^{\prime }.$
\end{enumerate}
\end{theorem}

\begin{notation}
If $T\in \mathcal{B}_{aaa}^{\prime }$ and $T=P+Q$ on $\mathbb{J},$ since
this decomposition is unique by (\cite{BT-AAAD}, Proposition $12$), the
distribution $P~$is called the principal~term of $\ T,$ the distribution$~Q$%
~is called the corrective term of $T,$ we denote them respectively $T_{aa}$
and $T_{cor}.$ This is summarized by the notation $T=\left(
T_{aa}+T_{cor}\right) \in \mathcal{B}_{aaa}^{\prime }.$
\end{notation}

\section{Asymptotically almost automorphic generalized functions}

We introduce and study an algebra of asymptotically almost automorphic
generalized functions.

Let $I:=]0,1],(u_{\varepsilon })_{\varepsilon }\in \left( \mathcal{D}%
_{L^{p}}\left( \mathbb{I}\right) \right) ^{I},p\in \left[ 1,+\infty \right]
,m\in \mathbb{Z}$ and $k\in \mathbb{Z}_{+},$ then the notation
\begin{equation*}
\left\vert u_{\varepsilon }\right\vert _{k,p,\mathbb{I}}=O\left( \varepsilon
^{m}\right) ,\varepsilon \rightarrow 0.\Leftrightarrow \exists c>0,\exists
\varepsilon _{0}\in I,~\forall \varepsilon <\varepsilon _{0},~\left\vert
u_{\varepsilon }\right\vert _{k,p,\mathbb{I}}\leq c\varepsilon ^{m}.
\end{equation*}

\begin{definition}
The algebra of asymptotically almost automorphic generalized functions is
denoted and defined as the quotient algebra
\begin{equation*}
\mathcal{G}_{aaa}:=\dfrac{\mathcal{M}_{aaa}}{\mathcal{N}_{aaa}},
\end{equation*}%
where the space of moderate elements is denoted and defined by
\begin{equation*}
\mathcal{M}_{aaa}:=\left\{ \left( u_{\varepsilon }\right) _{\varepsilon }\in
\left( \mathcal{B}_{aaa}\right) ^{I}:\forall k\in \mathbb{Z}_{+},\exists
m\in \mathbb{Z}_{+},\;\left\vert u_{\varepsilon }\right\vert _{k,\infty ,%
\mathbb{R} }=O\left( \varepsilon ^{-m}\right) ,\;\varepsilon \rightarrow
0\right\} ,
\end{equation*}%
and the space of null elements by
\begin{equation*}
\mathcal{N}_{aaa}:=\left\{ \left( u_{\varepsilon }\right) _{\varepsilon }\in
\left( \mathcal{B}_{aaa}\right) ^{I}:\forall k\in \mathbb{Z}_{+},\forall
m\in \mathbb{Z}_{+},\;\left\vert u_{\varepsilon }\right\vert _{k,\infty ,%
\mathbb{R} }=O\left( \varepsilon ^{m}\right) ,\;\varepsilon \rightarrow
0\right\} .
\end{equation*}
\end{definition}

\begin{example}
We have $\mathcal{G}_{aap}\subsetneq \mathcal{G}_{aaa},$ where $\mathcal{G}%
_{aap}$ is the algebra of asymptotically almost periodic generalized
functions of \cite{BK-AAPGF}.
\end{example}

Some properties of $\mathcal{M}_{aaa}$ and $\mathcal{N}_{aaa}$ are given in
the following results.

\begin{proposition}
\label{M-N}

\begin{enumerate}
\item We have the null characterization of $\mathcal{N}_{aaa},$\ i.e.%
\begin{equation*}
\mathcal{N}_{aaa}=\left\{ \left( u_{\varepsilon }\right) _{\varepsilon }\in
\mathcal{M}_{aaa}:\forall m\in \mathbb{Z}_{+},\;\left\vert u_{\varepsilon
}\right\vert _{0,\infty ,%
\mathbb{R}
}=O\left( \varepsilon ^{m}\right) ,\varepsilon \rightarrow 0\right\}
\end{equation*}

\item The space $\mathcal{M}_{aaa}$ is an algebra stable under translation
and derivation.

\item The space $\mathcal{N}_{aaa}$ is an ideal of $\mathcal{M}_{aaa}.$
\end{enumerate}
\end{proposition}

\begin{proof}
\begin{enumerate}
\item The proof is based on the following Landau-Kolmogorov inequality
\begin{equation*}
\left\Vert f^{(p)}\right\Vert _{L^{\infty }\left( \mathbb{%
\mathbb{R}
}\right) }\;\leq 2\pi \left\Vert f\right\Vert _{L^{\infty }\left( \mathbb{%
\mathbb{R}
}\right) }^{1-\frac{p}{n}}\;\left\Vert f^{(n)}\right\Vert _{L^{\infty
}\left( \mathbb{%
\mathbb{R}
}\right) }^{\frac{p}{n}},
\end{equation*}%
where $0<p<n\in \mathbb{%
\mathbb{Z}
}_{+}$ and $f$ is of class $\mathcal{C}^{n}\left(
\mathbb{R}
\right) .$ \newline
Let $\left( u_{\varepsilon }\right) _{\varepsilon }\in \mathcal{M}_{aaa},$
i.e.
\begin{equation}
\forall k\in \mathbb{Z}_{+},~\exists m\in \mathbb{Z}_{+},~\exists
c>0,~\exists \varepsilon _{1}\in I,~\forall \varepsilon <\varepsilon
_{1},~\left\vert u_{\varepsilon }\right\vert _{k,\infty ,%
\mathbb{R}
}\leq c\varepsilon ^{-m},  \label{equ7}
\end{equation}%
and it satisfies the estimate of order zero, i.e.
\begin{equation}
\forall m_{2}\in \mathbb{Z}_{+},~\exists c_{2}>0,~\exists \varepsilon
_{2}\in I,~\forall \varepsilon <\varepsilon _{2},~\left\vert u_{\varepsilon
}\right\vert _{0,\infty ,%
\mathbb{R}
}\leq c_{2}\varepsilon ^{m_{2}}.  \label{equ8}
\end{equation}%
The Landau-Kolmogorov inequality for $p=j,~n=2j,$ (\ref{equ7}) and (\ref%
{equ8}) give $\forall k\in \mathbb{Z}_{+},$
\begin{eqnarray*}
\left\vert u_{\varepsilon }\right\vert _{k,\infty ,%
\mathbb{R}
} &\leq &\sum\limits_{j\leq k}2\pi \left\Vert u_{\varepsilon }\right\Vert
_{L^{\infty }\left(
\mathbb{R}
\right) }^{1-\frac{1}{2}}\;\left\Vert u_{\varepsilon }^{(2j)}\right\Vert
_{L^{\infty }\left(
\mathbb{R}
\right) }^{\frac{1}{2}}, \\
&\leq &2\pi \left( \left\vert u_{\varepsilon }\right\vert _{0,\infty ,%
\mathbb{R}
}\right) ^{\frac{1}{2}}\sum\limits_{l\leq 2k}\left\Vert u_{\varepsilon
}^{(l)}\right\Vert _{L^{\infty }\left(
\mathbb{R}
\right) }^{\frac{1}{2}}, \\
&\leq &2\pi c^{\frac{1}{2}}c_{2}^{\frac{1}{2}}\varepsilon ^{{\frac{-m}{2}}%
}\varepsilon ^{\frac{m_{2}}{2}}.
\end{eqnarray*}%
Taking $m_{2}\in \mathbb{Z}_{+}$ such that $m_{0}=2\left( -\frac{m}{2}+\frac{%
m_{2}}{2}\right) \in \mathbb{Z}_{+},$ then we obtain $\forall k\in \mathbb{Z}%
_{+},~\forall m_{0}\in \mathbb{Z}_{+},~\exists c=2\pi c^{\frac{1}{2}}c_{2}^{%
\frac{1}{2}}>0,~\exists \varepsilon _{0}=\inf \left( \varepsilon
_{1},\varepsilon _{2}\right) \in I,~\forall \varepsilon <\varepsilon
_{0},~\left\vert u_{\varepsilon }\right\vert _{k,\infty ,%
\mathbb{R}
}\leq c\varepsilon ^{m_{0}}.~$Which gives $\left( u_{\varepsilon }\right)
_{\varepsilon }\in \mathcal{N}_{aaa}.$

\item The stability under translation and derivation of the space $\mathcal{M%
}_{aaa}$ is obvious. Let $\left( u_{\varepsilon }\right) _{\varepsilon
},~\left( v_{\varepsilon }\right) _{\varepsilon }\in \mathcal{M}_{aaa},$
i.e. they satisfy (\ref{equ7}) and for $j\in \mathbb{Z}_{+},$
\begin{eqnarray*}
\left\Vert \left( u_{\varepsilon }v_{\varepsilon }\right) ^{(j)}\right\Vert
_{L^{\infty }\left(
\mathbb{R}
\right) }\; &\leq &\sum\limits_{i\leq j}\dfrac{j!}{i!(j-i)!}\left\Vert
u_{\varepsilon }^{(i)}\right\Vert _{L^{\infty }\left(
\mathbb{R}
\right) }\left\Vert v_{\varepsilon }^{(j-i)}\right\Vert _{L^{\infty }\left(
\mathbb{R}
\right) }, \\
&\leq &\left\vert u_{\varepsilon }\right\vert _{j,\infty ,%
\mathbb{R}
}\left\vert v_{\varepsilon }\right\vert _{j,\infty ,%
\mathbb{R}
}\sum\limits_{i\leq j}\dfrac{j!}{i!(j-i)!}, \\
&\leq &2^{j}c_{1}c_{2}\varepsilon ^{-m_{1}}\varepsilon ^{-m_{2}},
\end{eqnarray*}%
consequently, $\forall k\in \mathbb{Z}_{+},~\exists m=\left(
m_{1}+m_{2}\right) \in \mathbb{Z}_{+},~\left\vert u_{\varepsilon
}v_{\varepsilon }\right\vert _{k,\infty ,%
\mathbb{R}
}=O\left( \varepsilon ^{-m}\right) ,\varepsilon \rightarrow 0.$ Which shows
that $\left( u_{\varepsilon }v_{\varepsilon }\right) _{\varepsilon }\in
\mathcal{M}_{aaa}.$

\item Let $(v_{\varepsilon })_{\varepsilon }\in \mathcal{M}_{aaa}$ and $%
\left( u_{\varepsilon }\right) _{\varepsilon }\in \mathcal{N}_{aaa},$ i.e. $%
\left( v_{\varepsilon }\right) _{\varepsilon }$ satisfies (\ref{equ7}) and $%
\left( u_{\varepsilon }\right) _{\varepsilon }$ satisfies
\begin{equation}
\forall k\in \mathbb{Z}_{+},~\forall m^{^{\prime }}\in \mathbb{Z}%
_{+},~\exists c^{^{\prime }}>0,~\exists \varepsilon _{1}^{^{\prime }}\in
I,~\forall \varepsilon <\varepsilon _{1}^{^{\prime }},~\left\vert
u_{\varepsilon }\right\vert _{k,\infty ,%
\mathbb{R}
}\leq c^{^{\prime }}\varepsilon ^{m^{^{\prime }}}.  \label{equ9}
\end{equation}%
Since the family of the norms $\left\vert \cdot \right\vert _{k,\infty ,%
\mathbb{R}
}$ is compatible with the algebraic structure of $\mathcal{B},$ i.e. $%
\forall k\in \mathbb{Z_{+}},~\exists c_{k}>0$ such that
\begin{equation*}
\left\vert u_{\varepsilon }v_{\varepsilon }\right\vert _{k,\infty ,%
\mathbb{R}
}\leq c_{k}\left\vert u_{\varepsilon }\right\vert _{k,\infty ,%
\mathbb{R}
}\left\vert v_{\varepsilon }\right\vert _{k,\infty ,%
\mathbb{R}
}\leq c_{k}cc^{^{\prime }}\varepsilon ^{-m}\varepsilon ^{m^{^{\prime }}}.
\end{equation*}%
Take $m^{\prime }\in \mathbb{Z}_{+}$ such that $m_{0}=\left( -m+m^{^{\prime
}}\right) \in \mathbb{Z}_{+},$ so we obtain $\forall k\in \mathbb{Z}%
_{+},~\forall m_{0}\in \mathbb{Z}_{+},~\exists C=c_{k}cc^{^{\prime
}}>0,~\exists \varepsilon _{0}=\inf \left( \varepsilon _{1},\varepsilon
_{1}^{^{\prime }}\right) \in I,~\forall \varepsilon <\varepsilon
_{0},~\left\vert u_{\varepsilon }v_{\varepsilon }\right\vert _{k,\infty ,%
\mathbb{R}
}\leq C\varepsilon ^{m_{0}},$ which implies $\left( u_{\varepsilon
}v_{\varepsilon }\right) _{\varepsilon }\in \mathcal{N}_{aaa}.$
\end{enumerate}
\end{proof}

Let $\rho \in \mathcal{S}$ such that $\int\limits_{\mathbb{R}}\rho \left(
x\right) dx=1$ and $\int\limits_{\mathbb{R}}x^{k}\rho \left( x\right)
dx=0,\forall k\in \mathbb{N}.$ Set $\rho _{\varepsilon }\left( \cdot \right)
:=\dfrac{1}{\varepsilon }\rho \left( \frac{\cdot }{\varepsilon }\right)
,~\varepsilon >0.$

Define the following maps
\begin{eqnarray*}
&&%
\begin{array}{cccc}
\iota _{aaa}: & \mathcal{B}_{aaa}^{^{\prime }} & \longrightarrow & \mathcal{G%
}_{aaa} \\
& T & \longmapsto & \left( T\ast \rho _{\varepsilon }\right) _{\varepsilon }+%
\mathcal{N}_{aaa}%
\end{array}
\\
&&%
\begin{array}{cccc}
\sigma _{aaa}: & \mathcal{B}_{aaa} & \longrightarrow & \mathcal{G}_{aaa} \\
& f & \longmapsto & \left( f\right) _{\varepsilon }+\mathcal{N}_{aaa}%
\end{array}
\\
&&%
\begin{array}{cccc}
\gamma _{aaa}: & \mathcal{B}_{aaa} & \longrightarrow & \mathcal{B}%
_{aaa}^{\prime } \\
& f & \longmapsto & f%
\end{array}%
\end{eqnarray*}

\begin{proposition}
The following diagram of linear embeddings
\begin{equation*}
\begin{array}{lll}
\mathcal{B}_{aaa} & \overset{\gamma _{aaa}}{\longrightarrow } & \mathcal{B}%
_{aaa}^{\prime } \\
& \sigma _{aaa}\searrow & \text{ }\downarrow \iota _{aaa} \\
&  & \mathcal{G}_{aaa}%
\end{array}%
\end{equation*}%
is commutative.
\end{proposition}

\begin{proof}
For $f\in \mathcal{B}_{aaa},$ we have $\gamma _{aaa}\left( f\right) \in
\mathcal{B}_{aaa}\subset \mathcal{C}_{aaa}.$ We conclude from (\cite{BT-AAAD}%
, Example $3-\left( 1\right) $) that $\gamma _{aaa}\left( f\right) \in
\mathcal{B}_{aaa}^{^{\prime }}.$ Let $T\in \mathcal{B}_{aaa}^{^{\prime }},$
by the characterization of an asymptotically almost automorphic distribution
$\exists (f_{i})_{i\leq m}\subset \mathcal{C}_{aaa}$ such that \ $%
T=\sum\limits_{i\leq m}f_{i}^{(i)}.$ If $j\in \mathbb{Z_{+}},$
\begin{eqnarray*}
\left\vert \left( T\ast \rho _{\varepsilon }\right) ^{(j)}(x)\right\vert
&\leq &\sum\limits_{i\leq m}\dfrac{1}{\varepsilon ^{i+j}}\int\limits_{%
\mathbb{R}}\left\vert f_{i}\left( x-\varepsilon y\right) \rho ^{(i+j)}\left(
y\right) \right\vert dy\;, \\
&\leq &\sum\limits_{i\leq m}\dfrac{1}{\varepsilon ^{i+j}}\Vert f_{i}\Vert
_{L^{\infty }\left( \mathbb{R} \right) }\int\limits_{\mathbb{R}}\left\vert
\rho ^{(i+j)}\left( y\right) \right\vert dy,
\end{eqnarray*}%
consequently, there exists $c_{m,j}>0$ such that
\begin{equation*}
\left\Vert \left( T\ast \rho _{\varepsilon }\right) ^{(j)}\right\Vert
_{L^{\infty }\left( \mathbb{R} \right) }\leq \dfrac{c_{m,j}}{\varepsilon
^{m+j}},
\end{equation*}%
hence, $\forall k\in \mathbb{Z} _{+},\exists m_{1}\in \mathbb{Z} _{+},$
\begin{equation*}
\left\vert T\ast \rho _{\varepsilon }\right\vert _{k,\infty ,\mathbb{R}
}=O\left( \varepsilon ^{-m_{1}}\right) ,~\varepsilon \rightarrow 0,
\end{equation*}%
which gives that $\left( T\ast \rho _{\varepsilon }\right) _{\varepsilon
}\in \mathcal{M}_{aaa}.$ The linearity of $\iota _{aaa}$ results from the
fact that the convolution is linear. If $\left( T\ast \rho _{\varepsilon
}\right) _{\varepsilon }\in \mathcal{N}_{aaa},$ then
\begin{equation}
\forall m^{^{\prime }}\in \mathbb{Z}_{+},~\exists c^{^{\prime }}>0,~\exists
\varepsilon _{1}^{^{\prime }}\in I,~\forall \varepsilon <\varepsilon
_{1}^{^{\prime }},~\left\Vert T\ast \rho _{\varepsilon }\right\Vert
_{L^{\infty }\left( \mathbb{R} \right) }\leq c^{^{\prime }}\varepsilon
^{m^{^{\prime }}}.  \label{embedding*}
\end{equation}%
Due to regularization we have
\begin{equation*}
\left\langle T,\psi \right\rangle =\lim\limits_{\varepsilon \rightarrow
0}\int \left( T\ast \rho _{\varepsilon }\right) \left( x\right) \psi \left(
x\right) dx,\psi \in \mathcal{D}_{L^{1}}\left( \mathbb{R} \right) .
\end{equation*}%
From (\ref{embedding*}), it holds $\forall m^{^{\prime }}\in \mathbb{Z}%
_{+},~\exists c^{^{\prime }}>0,~\exists \varepsilon _{1}^{^{\prime }}\in
I,\forall \varepsilon <\varepsilon _{1}^{^{\prime }},$
\begin{equation*}
\left\vert \int (T\ast \rho _{\varepsilon })(x)\psi (x)dx\right\vert \leq
\left\Vert \psi \right\Vert _{L^{1}\left( \mathbb{R} \right) }c^{^{\prime
}}\varepsilon ^{m^{^{\prime }}},
\end{equation*}%
consequently, when $\varepsilon \rightarrow 0$ we obtain $\left\langle
T,\psi \right\rangle =0,~\forall \psi \in \mathcal{D}_{L^{1}}\left( \mathbb{R%
} \right) ,$ therefore $\iota _{aaa}$ is injective. Finally,
\begin{eqnarray*}
\iota _{aaa}\left( T^{(j)}\right) &=&\left( T^{(j)}\ast \rho _{\varepsilon
}\right) _{\varepsilon }+\mathcal{N}_{aaa} \\
&=&\left( T\ast \rho _{\varepsilon }\right) _{\varepsilon }^{(j)}+\mathcal{N}%
_{aaa}=\left( \iota _{aaa}(T)\right) ^{(j)},~\forall j\in \mathbb{Z}_{+},
\end{eqnarray*}%
i.e. the embedding $\iota _{aaa}$ commutes with derivatives.

Let $f\in \mathcal{B}_{aaa},$ we have to show that $\left( f\ast \rho
_{\varepsilon }-f\right) _{\varepsilon }\in \mathcal{N}_{aaa}.$ By Taylor's
formula, for $\theta \in ]0,1[$ and $m\in \mathbb{N},$ we set for $j\in
\mathbb{Z} _{+},$ that $\left( f^{(j)}\ast \rho _{\varepsilon
}-f^{(j)}\right) \left( x\right) $ equals
\begin{equation*}
\int\limits_{\mathbb{R}}\sum\limits_{l=1}^{m-1}\dfrac{\left( -\varepsilon
y\right) ^{l}}{l!}f^{(l+j)}\left( x\right) \rho \left( y\right)
dy+\int\limits_{\mathbb{R} }\dfrac{\left( -\varepsilon y\right) ^{m}}{m!}%
f^{(m+j)}\left( x-\theta \varepsilon y\right) \rho \left( y\right) dy,
\end{equation*}%
and then
\begin{eqnarray*}
\left\Vert f^{(j)}\ast \rho _{\varepsilon }-f^{(j)}\right\Vert _{L^{\infty
}\left( \mathbb{R} \right) }\; &\leq &\;\dfrac{\varepsilon ^{m}}{m!}%
\sup_{x\in \mathbb{R}}\int\limits_{\mathbb{R} }\left\vert \left( -y\right)
^{m}\right\vert \left\vert f^{(m+j)}\left( x-\theta \varepsilon y\right)
\right\vert \left\vert \rho \left( y\right) \right\vert dy, \\
&\leq &\;\dfrac{\varepsilon ^{m}}{m!}\;\left\Vert f^{(m+j)}\right\Vert
_{L^{\infty }\left( \mathbb{R} \right) }\left\Vert y^{m}\rho \right\Vert
_{L^{1}\left( \mathbb{R} \right) }.
\end{eqnarray*}%
Hence, $\forall k\in \mathbb{Z} _{+},\forall m\in \mathbb{N} ,$
\begin{equation*}
\left\vert f\ast \rho _{\varepsilon }-f\right\vert _{k,\infty ,\mathbb{R}
}\leq \;\dfrac{\varepsilon ^{m}}{m!}\;\left\vert f\right\vert _{m+k,\infty ,%
\mathbb{R} }\;\left\Vert y^{m}\rho \right\Vert _{L^{1}\left( \mathbb{R}
\right) },
\end{equation*}%
which means that $\left( f\ast \rho _{\varepsilon }-f\right) _{\varepsilon
}\in \mathcal{N}_{aaa}.$
\end{proof}

Recall the algebra of almost automorphic generalized functions, see \cite%
{BKT-AAGF}, denoted and defined by
\begin{equation*}
\mathcal{G}_{aa}:=\dfrac{\mathcal{M}_{aa}}{\mathcal{N}_{aa}},
\end{equation*}%
where
\begin{equation*}
\mathcal{M}_{aa}:=\left\{ (u_{\varepsilon })_{\varepsilon }\in \left(
\mathcal{B}_{aa}\right) ^{I}:\forall k\in \mathbb{Z}_{+},\exists m\in
\mathbb{Z}_{+},\;\left\vert u_{\varepsilon }\right\vert _{k,\infty ,\mathbb{R%
} }=O\left( \varepsilon ^{-m}\right) ,\;\varepsilon \rightarrow 0\right\} .
\end{equation*}%
\begin{equation*}
\mathcal{N}_{aa}:=\left\{ (u_{\varepsilon })_{\varepsilon }\in (\mathcal{B}%
_{aa})^{I}:\forall k\in \mathbb{Z}_{+},\forall m\in \mathbb{Z}%
_{+},\;\left\vert u_{\varepsilon }\right\vert _{k,\infty ,\mathbb{R}
}=O\left( \varepsilon ^{m}\right) ,\;\varepsilon \rightarrow 0\right\} .
\end{equation*}%
The algebra of $L^{p}$-generalized functions on $\mathbb{I},$ $1\leq p\leq
+\infty ,$ see \cite{G}, is defined as the quotient algebra%
\begin{equation*}
\mathcal{G}_{L^{p}}\left( \mathbb{I}\right) :=\dfrac{\mathcal{M}%
_{L^{p}}\left( \mathbb{I}\right) }{\mathcal{N}_{L^{p}}\left( \mathbb{I}%
\right) },
\end{equation*}%
where
\begin{equation*}
\mathcal{M}_{L^{p}}\left( \mathbb{I}\right) :=\left\{ (u_{\varepsilon
})_{\varepsilon }\in \left( \mathcal{D}_{L^{p}}\left( \mathbb{I}\right)
\right) ^{I}:\forall k\in \mathbb{Z}_{+},\exists m>0,\;\left\vert
u_{\varepsilon }\right\vert _{k,p,\mathbb{I}}=O\left( \varepsilon
^{-m}\right) ,\;\varepsilon \rightarrow 0\right\} .
\end{equation*}%
\begin{equation*}
\mathcal{N}_{L^{p}}\left( \mathbb{I}\right) :=\left\{ (u_{\varepsilon
})_{\varepsilon }\in \left( \mathcal{D}_{L^{p}}\left( \mathbb{I}\right)
\right) ^{I}:\forall k\in \mathbb{Z}_{+},\forall m>0,\;\left\vert
u_{\varepsilon }\right\vert _{k,p,\mathbb{I}}=O\left( \varepsilon
^{m}\right) ,\;\varepsilon \rightarrow 0\right\} .
\end{equation*}

\begin{notation}
Denote $\mathcal{G}_{L^{\infty }}\left( \mathbb{I}\right) :=\mathcal{G}_{%
\mathcal{B}}\left( \mathbb{I}\right) ,\mathcal{G}_{\mathcal{B}}:=\mathcal{G}%
_{\mathcal{B}}\left( \mathbb{R} \right) $ and $\mathcal{G}_{L^{1}}:=\mathcal{%
G}_{L^{1}}\left( \mathbb{R} \right) .$
\end{notation}

For $\omega \in \mathbb{R},$ the translate $\tau _{\omega }\widetilde{u}$ of
$\widetilde{u}=\left[ (u_{\varepsilon })_{\varepsilon }\right] \in \mathcal{G%
}_{\mathcal{B}}$ is defined by
\begin{equation*}
\tau _{\omega }\widetilde{u}:=\left[ \left( \tau _{\omega }u_{\varepsilon
}\right) _{\varepsilon }\right] .
\end{equation*}%
For $j\in \mathbb{\mathbb{Z} }_{+},$ the derivation $\widetilde{u}%
^{^{^{(j)}}}$ of $\widetilde{u}$ is defined by
\begin{equation*}
\widetilde{u}^{^{^{(j)}}}:=\left[ \left( u_{\varepsilon }^{^{(j)}}\right)
_{\varepsilon }\right] .
\end{equation*}%
Let $\widetilde{v}=[(v_{\varepsilon })_{\varepsilon }]\in \mathcal{G}_{%
\mathcal{B}},$ the product $\widetilde{u}\times \widetilde{v}$ is defined by
\begin{equation*}
\widetilde{u}\times \widetilde{v}:=\left[ (u_{\varepsilon }v_{\varepsilon
})_{\varepsilon }\right] .
\end{equation*}%
Let $\widetilde{v}=\left[ (v_{\varepsilon })_{\varepsilon }\right] \in
\mathcal{G}_{L^{1}},$ the convolution $\widetilde{u}\ast \widetilde{v}$ is
defined by
\begin{equation*}
\widetilde{u}\ast \widetilde{v}:=\left[ \left( u_{\varepsilon }\ast
v_{\varepsilon }\right) _{\varepsilon }\right] .
\end{equation*}

The following results lift the results of Proposition \ref{prop1.1} to $%
\mathcal{G}_{aaa}.$

\begin{proposition}
\label{prop3.2}

\begin{enumerate}
\item $\mathcal{G}_{aaa}$ is a subalgebra of $\mathcal{G}_{\mathcal{B}}$
stable under translation and derivation.

\item $\mathcal{G}_{aaa}\times \mathcal{G}_{aa}\subset \mathcal{G}_{aaa}.$

\item $\mathcal{G}_{aaa}\ast \mathcal{G}_{L^{1}}\subset \mathcal{G}_{aaa}.$
\end{enumerate}
\end{proposition}

\begin{proof}
\begin{enumerate}
\item From Proposition \ref{M-N}-$\left( 2\right) ,$ we deduce that $%
\mathcal{G}_{aaa}$ is an algebra stable under translation and derivation.
Let $(u_{\varepsilon })_{\varepsilon }\in \mathcal{M}_{aaa},$ satisfies (\ref%
{equ7}) and as $\mathcal{B}_{aaa}\subset \mathcal{B},$ hence $%
(u_{\varepsilon })_{\varepsilon }\in \mathcal{M}_{\mathcal{B}}.$ If $%
(u_{\varepsilon })_{\varepsilon }\in \mathcal{N}_{aaa},$ in the same way we
prove that $(u_{\varepsilon })_{\varepsilon }\in \mathcal{N}_{\mathcal{B}}.$

\item Let $\widetilde{u}=[(u_{\varepsilon })_{\varepsilon }]\in \mathcal{G}%
_{aaa}$ and $\widetilde{v}=[(v_{\varepsilon })_{\varepsilon }]\in \mathcal{G}%
_{aa}.$ As $(u_{\varepsilon })_{\varepsilon }\in \mathcal{M}_{aaa}$ and $%
(v_{\varepsilon })_{\varepsilon }\in \mathcal{M}_{aa},$ so $(u_{\varepsilon
})_{\varepsilon }$ and $(v_{\varepsilon })_{\varepsilon }$ satisfy the
estimate (\ref{equ7}). In view of Proposition \ref{prop1.1}-$\left( 2\right)
$ it follows that $\forall \varepsilon \in I,~u_{\varepsilon }v_{\varepsilon
}\in \mathcal{B}_{aaa},$ and for every $j\in \mathbb{Z}_{+},$
\begin{equation*}
\Vert \left( u_{\varepsilon }v_{\varepsilon }\right) ^{(j)}\Vert _{L^{\infty
}\left( \mathbb{R} \right) }\leq 2^{j}\left\vert u_{\varepsilon }\right\vert
_{j,\infty ,\mathbb{R} }\left\vert v_{\varepsilon }\right\vert _{j,\infty ,%
\mathbb{R} }\leq 2^{j}c_{1}c_{2}\varepsilon ^{-m_{1}}\varepsilon ^{-m_{2}},
\end{equation*}%
therefore, $\forall k\in \mathbb{Z}_{+},~\exists m=\left( m_{1}+m_{2}\right)
\in \mathbb{Z}_{+},~\left\vert u_{\varepsilon }v_{\varepsilon }\right\vert
_{k,\infty ,\mathbb{R} }=O(\varepsilon ^{-m}),\;\varepsilon \rightarrow 0,$
i.e. $\left( u_{\varepsilon }v_{\varepsilon }\right) _{\varepsilon }\in
\mathcal{M}_{aaa}. $ The product $\widetilde{u}\times \widetilde{v}$ is
independent on the representatives. Indeed, suppose that $(u_{\varepsilon
}^{^{\prime }})_{\varepsilon }\in \mathcal{M}_{aaa}$ and $(v_{\varepsilon
}^{^{\prime }})_{\varepsilon }\in \mathcal{M}_{aa}$ are others
representatives of $\widetilde{u}$ and $\widetilde{v}$ respectively, for $%
j\in \mathbb{Z}_{+},$
\begin{eqnarray*}
\left\Vert \left( u_{\varepsilon }v_{\varepsilon }-u_{\varepsilon
}^{^{\prime }}v_{\varepsilon }^{^{\prime }}\right) ^{(j)}\right\Vert
_{L^{\infty }\left( \mathbb{R} \right) } &=&\left\Vert \left( u_{\varepsilon
}v_{\varepsilon }-u_{\varepsilon }^{^{\prime }}v_{\varepsilon
}+u_{\varepsilon }^{^{\prime }}v_{\varepsilon }-u_{\varepsilon }^{^{\prime
}}v_{\varepsilon }^{^{\prime }}\right) ^{(j)}\right\Vert _{L^{\infty }\left(
\mathbb{R} \right) }, \\
&\leq &\left\Vert \left( (u_{\varepsilon }-u_{\varepsilon }^{^{\prime
}})v_{\varepsilon }\right) ^{(j)}\right\Vert _{L^{\infty }\left( \mathbb{R}
\right) }+\left\Vert \left( u_{\varepsilon }^{^{\prime }}(v_{\varepsilon
}-v_{\varepsilon }^{^{\prime }}\right) ^{(j)}\right\Vert _{L^{\infty }\left(
\mathbb{R} \right) }, \\
&\leq &2^{j}\left( \left\vert u_{\varepsilon }-u_{\varepsilon }^{^{\prime
}}\right\vert _{j,\infty ,\mathbb{R} }\left\vert v_{\varepsilon }\right\vert
_{j,\infty ,\mathbb{R} }+\left\vert u_{\varepsilon }^{^{\prime }}\right\vert
_{j,\infty ,\mathbb{R} }\left\vert v_{\varepsilon }-v_{\varepsilon
}^{^{\prime }}\right\vert _{j,\infty ,\mathbb{R} }\right) ,
\end{eqnarray*}%
since $(u_{\varepsilon }-u_{\varepsilon }^{^{\prime }})_{\varepsilon }\in
\mathcal{N}_{aaa},$ $(v_{\varepsilon }-v_{\varepsilon }^{^{\prime
}})_{\varepsilon }\in \mathcal{N}_{aa},$ then $\forall k\in \mathbb{Z}
_{+},\forall m\in \mathbb{Z}_{+},\left\vert u_{\varepsilon }v_{\varepsilon
}-u_{\varepsilon }^{^{\prime }}v_{\varepsilon }^{^{\prime }}\right\vert
_{k,\infty ,\mathbb{R} }=O\left( \varepsilon ^{m}\right) ,$ as $\varepsilon
\rightarrow 0,$ which implies that $\left( u_{\varepsilon }v_{\varepsilon
}-u_{\varepsilon }^{^{\prime }}v_{\varepsilon }^{^{\prime }}\right)
_{\varepsilon }\in \mathcal{N}_{aaa}.$

\item Let $\widetilde{u}=[(u_{\varepsilon })_{\varepsilon }]\in \mathcal{G}%
_{aaa}$ and $\widetilde{v}=[(v_{\varepsilon })_{\varepsilon }]\in \mathcal{G}%
_{L^{1}},$ so $(u_{\varepsilon })_{\varepsilon }\in \mathcal{M}_{aaa}$ and $%
(v_{\varepsilon })_{\varepsilon }\in \mathcal{M}_{L^{1}},$ that is $%
(u_{\varepsilon })_{\varepsilon }$ satisfies the estimate $\left( \ref{equ7}%
\right) $, \ and $(v_{\varepsilon })_{\varepsilon }$ satisfies
\begin{equation*}
\forall k\in \mathbb{Z}_{+},~\exists m_{1}>0,~\exists c>0,~\exists
\varepsilon _{0}\in I,~\forall \varepsilon <\varepsilon _{0},~\left\vert
v_{\varepsilon }\right\vert _{k,1,\mathbb{R} }\leq c\varepsilon ^{-m_{1}}.
\end{equation*}%
Proposition \ref{prop1.1}-$\left( 3\right) $ gives $\forall \varepsilon \in
I,~u_{\varepsilon }\ast v_{\varepsilon }\in \mathcal{B}_{aaa}.$ Due to Young
inequality we obtain
\begin{equation*}
\Vert (u_{\varepsilon }\ast v_{\varepsilon })^{(j)}\Vert _{L^{\infty }\left(
\mathbb{R} \right) }\;\leq \left\Vert u_{\varepsilon }^{(j)}\right\Vert
_{L^{\infty }\left( \mathbb{R} \right) }\left\Vert v_{\varepsilon
}\right\Vert _{L^{1}\left( \mathbb{R} \right) }.
\end{equation*}%
For every $k\in \mathbb{Z}_{+},$
\begin{equation*}
\left\vert u_{\varepsilon }\ast v_{\varepsilon }\right\vert _{k,\infty ,%
\mathbb{R} }\;\leq \sum\limits_{j\leq k}\left\Vert u_{\varepsilon
}^{(j)}\right\Vert _{L^{\infty }\left( \mathbb{R} \right) }\left\Vert
v_{\varepsilon }\right\Vert _{L^{1}\left( \mathbb{R} \right) } \\
\leq \left\vert u_{\varepsilon }\right\vert _{k,\infty ,\mathbb{R}
}\left\vert v_{\varepsilon }\right\vert _{0,1,\mathbb{R} },
\end{equation*}%
consequently, $\forall k\in \mathbb{Z} _{+},\exists m\in \mathbb{Z}
_{+},\left\vert u_{\varepsilon }\ast v_{\varepsilon }\right\vert _{k,\infty ,%
\mathbb{R} }=O(\varepsilon ^{-m}),\;\varepsilon \rightarrow 0,$ so $%
(u_{\varepsilon }\ast v_{\varepsilon })_{\varepsilon }\in \mathcal{M}_{aaa}.$
The convolution $\widetilde{u}\ast \widetilde{v}$ does not depend on the
representatives. Indeed, suppose that $(u_{\varepsilon }^{^{\prime
}})_{\varepsilon }\in \mathcal{M}_{aaa}$ and $(v_{\varepsilon }^{^{\prime
}})_{\varepsilon }\in \mathcal{M}_{L^{1}}$ are others representatives of $%
\widetilde{u}$ and $\widetilde{v}$ respectively, for every $k\in \mathbb{Z}
_{+},$
\begin{eqnarray*}
\left\vert u_{\varepsilon }\ast v_{\varepsilon }-u_{\varepsilon }^{^{\prime
}}\ast v_{\varepsilon }^{^{\prime }}\right\vert _{k,\infty ,%
\mathbb{R}
\mathbb{R} } &=&\left\vert u_{\varepsilon }\ast v_{\varepsilon
}-u_{\varepsilon }^{^{\prime }}\ast v_{\varepsilon }+u_{\varepsilon
}^{^{\prime }}\ast v_{\varepsilon }-u_{\varepsilon }^{^{\prime }}\ast
v_{\varepsilon }^{^{\prime }}\right\vert _{k,\infty ,\mathbb{R} }, \\
&\leq &\left\vert u_{\varepsilon }-u_{\varepsilon }^{^{\prime }}\right\vert
_{k,\infty ,\mathbb{R} }\left\vert v_{\varepsilon }\right\vert _{0,1,\mathbb{%
R} }+\left\vert u_{\varepsilon }^{^{\prime }}\right\vert _{k,\infty ,\mathbb{%
R} }\left\vert v_{\varepsilon }-v_{\varepsilon }^{^{\prime }}\right\vert
_{0,1,\mathbb{R} },
\end{eqnarray*}%
as $\left( u_{\varepsilon }-u_{\varepsilon }^{^{\prime }}\right)
_{\varepsilon }\in \mathcal{N}_{aaa},$ $\left( v_{\varepsilon
}-v_{\varepsilon }^{^{\prime }}\right) _{\varepsilon }\in \mathcal{N}%
_{L^{1}},$ then $\forall k\in \mathbb{Z} _{+},\forall m\in \mathbb{Z}%
_{+},~\left\vert u_{\varepsilon }\ast v_{\varepsilon }-u_{\varepsilon
}^{^{\prime }}\ast v_{\varepsilon }^{^{\prime }}\right\vert _{k,\infty ,%
\mathbb{R} }=O(\varepsilon ^{m}),$ as $\varepsilon \rightarrow 0,$ which
means that $\left( u_{\varepsilon }\ast v_{\varepsilon }-u_{\varepsilon
}^{^{\prime }}\ast v_{\varepsilon }^{^{\prime }}\right) _{\varepsilon }\in
\mathcal{N}_{aaa}.$
\end{enumerate}
\end{proof}

\section{A generalized Seeley theorem}

We give a result on the extension operators in the context of generalized
functions. It is needed in the proof of the decomposition of an
asymptotically almost automorphic generalized function. Let's first recall a
technical Lemma, see \cite{Seeley}.

\begin{lemma}
\label{Lemma-Seeley} There are two sequences of real numbers $%
\left(a_{l}\right)_{l\in\mathbb{Z}_{+}}$ and $\left(b_{l}\right)_{l\in
\mathbb{Z}_{+}}$ such that

\begin{enumerate}
\item $b_{l}<0,~\forall l\in \mathbb{Z}_{+}.$

\item $\sum\limits_{l=0}^{+\infty }\left\vert a_{l}\right\vert \left\vert
b_{l}\right\vert ^{n}<+\infty ,~\forall n\in \mathbb{Z}_{+}.$

\item $\sum\limits_{l=0}^{+\infty }a_{l}b_{l}^{n}=1,~\forall n\in \mathbb{Z}
_{+}.$

\item $b_{l}\rightarrow -\infty ,~l\rightarrow +\infty .$
\end{enumerate}
\end{lemma}

Define the space
\begin{equation*}
\mathcal{B}_{+,0}\left( \mathbb{I}\right) :=\left\{ \varphi \in \mathcal{B}(%
\mathbb{I}):\forall j\in \mathbb{Z}_{+},~\lim\limits_{x\rightarrow +\infty
}\varphi ^{(j)}(x)=0\right\} .
\end{equation*}%
The algebra of bounded generalized functions vanishing at infinity on $%
\mathbb{I}$ is defined by
\begin{equation*}
\mathcal{G}_{+,0}\left( \mathbb{I}\right) :=\dfrac{\mathcal{M}_{+,0}\left(
\mathbb{I}\right) }{\mathcal{N}_{+,0}\left( \mathbb{I}\right) },
\end{equation*}%
where
\begin{equation*}
\mathcal{M}_{+,0}\left( \mathbb{I}\right) :=\left\{ (u_{\varepsilon
})_{\varepsilon }\in \left( \mathcal{B}_{+,0}\left( \mathbb{I}\right)
\right) ^{I}:\forall k\in \mathbb{Z}_{+},~\exists m\in \mathbb{Z}%
_{+},\;\left\vert u_{\varepsilon }\right\vert _{k,\infty ,\mathbb{I}%
}=O(\varepsilon ^{-m}),\;\varepsilon \rightarrow 0\right\} .
\end{equation*}%
\begin{equation*}
\mathcal{N}_{+,0}\left( \mathbb{I}\right) :=\left\{ (u_{\varepsilon
})_{\varepsilon }\in \left( \mathcal{B}_{+,0}\left( \mathbb{I}\right)
\right) ^{I}:\forall k\in \mathbb{Z}_{+},~\forall m\in \mathbb{Z}%
_{+},\;\left\vert u_{\varepsilon }\right\vert _{k,\infty ,\mathbb{I}%
}=O(\varepsilon ^{m}),\;\varepsilon \rightarrow 0\right\} .
\end{equation*}

\begin{theorem}
\label{THM2} The linear extension operator $\widetilde{E}:\mathcal{G}_{%
\mathcal{B}}\left( \mathbb{J}\right) \longrightarrow \mathcal{G}_{\mathcal{B}%
}\left( \mathbb{R} \right) ,$ $\widetilde{u}=[(u_{\varepsilon
})_{\varepsilon }]\longmapsto \widetilde{E}\widetilde{u}=\left[ \left(
Eu_{\varepsilon }\right) _{\varepsilon }\right] ,$ where
\begin{equation*}
Eu_{\varepsilon }(x):=\left\{
\begin{array}{ccc}
u_{\varepsilon }\left( x\right) , & \ \text{if}\ \  & x\geq 0, \\
\sum\limits_{l=0}^{+\infty }a_{l}u_{\varepsilon }\left( b_{l}x\right) , &
\text{if} & x<0,%
\end{array}%
\right.
\end{equation*}%
is well defined and we have $\widetilde{E}\widetilde{u}_{\mid \mathbb{J}}=%
\widetilde{u}.$ In particular, $\forall \widetilde{u}\in \mathcal{G}%
_{+,0}\left( \mathbb{J}\right) ,\widetilde{E}\widetilde{u}\in \mathcal{G}%
_{+,0}\left( \mathbb{R} \right) .$
\end{theorem}

\begin{proof}
Let $\widetilde{u}=[(u_{\varepsilon })_{\varepsilon }]\in \mathcal{G}_{%
\mathcal{B}}\left( \mathbb{J}\right) ,$ and $(u_{\varepsilon })_{\varepsilon
}\in \mathcal{M}_{\mathcal{B}}\left( \mathbb{J}\right) $ be a representative
of $\widetilde{u}.$ So $\forall \varepsilon \in I,~Eu_{\varepsilon }\in
\mathcal{B}\left( \mathbb{R}\right) $ and $Eu_{\varepsilon \mid \mathbb{J}%
}=u_{\varepsilon }.$ Indeed, if $x<0,$ then $b_{l}x>0,~\forall l\in \mathbb{Z%
}_{+}$ in view of \ Lemma \ref{Lemma-Seeley}-$\left( 1\right) .$ Moreover
according to Lemma \ref{Lemma-Seeley}-$\left( 2\right) $ and as $%
u_{\varepsilon }\in \mathcal{B}\left( \mathbb{J}\right) ,\forall \varepsilon
\in I,$ hence $\forall n\in \mathbb{Z}_{+},\forall \varepsilon \in I,\forall
x<0,$
\begin{equation}
\left\vert \left( Eu_{\varepsilon }\right) ^{(n)}(x)\right\vert \leq
\left\Vert u_{\varepsilon }^{(n)}\right\Vert _{L^{\infty }\left( \mathbb{J}%
\right) }\sum_{l=0}^{+\infty }\left\vert a_{l}\right\vert \left\vert
b_{l}\right\vert ^{n}<+\infty ,  \label{equ9.7}
\end{equation}%
consequently, $\forall n\in \mathbb{Z} _{+},$ the series%
\begin{equation*}
\left( Eu_{\varepsilon }\right) ^{(n)}\left( x\right)
=\sum\limits_{l=0}^{+\infty }a_{l}b_{l}^{n}u_{\varepsilon }^{(n)}\left(
b_{l}x\right) ,\varepsilon \in I,
\end{equation*}%
absolutely converge. Furthermore due to Lemma \ref{Lemma-Seeley}-$\left(
3\right) $,
\begin{eqnarray*}
\lim\limits_{\substack{ x\rightarrow 0  \\ <}}\left( Eu_{\varepsilon
}\right) ^{(n)}\left( x\right) &=&\sum\limits_{l=0}^{+\infty
}a_{l}b_{l}^{n}\lim\limits_{\substack{ x\rightarrow 0  \\ <}}u_{\varepsilon
}^{(n)}(b_{l}x), \\
&=&u_{\varepsilon }^{(n)}\left( 0\right) \sum\limits_{l=0}^{+\infty
}a_{l}b_{l}^{n}=u_{\varepsilon }^{(n)}(0),
\end{eqnarray*}%
so $\forall \varepsilon \in I,Eu_{\varepsilon }\in \mathcal{E}\left( \mathbb{%
R}\right) .$ As $\forall \varepsilon \in I,u_{\varepsilon }\in \mathcal{B}%
\left( \mathbb{J}\right) $ and by ($\ref{equ9.7}$), it follows that $\forall
n\in \mathbb{Z}_{+},\forall \varepsilon \in I,\left( Eu_{\varepsilon
}\right) ^{(n)}\in L^{\infty }\left( \mathbb{R}\right) .$ In order to show
that $(Eu_{\varepsilon })_{\varepsilon }\in \mathcal{M}_{\mathcal{B}}\left(
\mathbb{R}
\right) ,$ we prove the following estimate
\begin{equation}
\forall \varepsilon \in I,~\forall k\in \mathbb{Z}_{+},~\exists
C_{k}>0,~\left\vert Eu_{\varepsilon }\right\vert _{k,\infty ,\mathbb{R}}\leq
C_{k}\left\vert u_{\varepsilon }\right\vert _{k,\infty ,\mathbb{J}}.
\label{equ9.8}
\end{equation}%
Indeed, it is clear that
\begin{equation*}
\left\Vert \left( Eu_{\varepsilon }\right) ^{(n)}\right\Vert _{L^{\infty
}\left( \mathbb{J}\right) }=\left\Vert u_{\varepsilon }^{(n)}\right\Vert
_{L^{\infty }(\mathbb{J})},
\end{equation*}%
and the estimate ($\ref{equ9.7}$), \ gives
\begin{equation*}
\left\Vert \left( Eu_{\varepsilon }\right) ^{(n)}\right\Vert _{L^{\infty
}\left( \mathbb{R} \setminus \mathbb{J}\right) }\leq \left\Vert
u_{\varepsilon }^{(n)}\right\Vert _{L^{\infty }\left( \mathbb{J}\right)
}\sum\limits_{l=0}^{+\infty }\left\vert a_{l}\right\vert \left\vert
b_{l}\right\vert ^{n},
\end{equation*}%
therefore%
\begin{equation*}
\left\Vert \left( Eu_{\varepsilon }\right) ^{(n)}\right\Vert _{L^{\infty
}\left( \mathbb{R} \right) }\leq C_{n}\left\Vert u_{\varepsilon
}^{(n)}\right\Vert _{L^{\infty }\left( \mathbb{J}\right) },
\end{equation*}%
where $C_{n}=\max \left( 1,\sum\limits_{l=0}^{+\infty }\left\vert
a_{l}\right\vert \left\vert b_{l}\right\vert ^{n}\right) .$ So $\forall k\in
\mathbb{Z}_{+},~\exists C_{k}:=\sum\limits_{n\leq k}C_{n}<+\infty ,~\forall
\varepsilon \in I,~u_{\varepsilon }\in \mathcal{B}\left( \mathbb{J}\right) ,$
\begin{equation}
\left\vert Eu_{\varepsilon }\right\vert _{k,\infty ,\mathbb{R}}\leq
C_{k}\left\vert u_{\varepsilon }\right\vert _{k,\infty ,\mathbb{J}},
\label{continuity}
\end{equation}%
this implies $(Eu_{\varepsilon })_{\varepsilon }\in \mathcal{M}_{\mathcal{B}%
}\left( \mathbb{R} \right) .~$The definition of $\widetilde{E}\widetilde{u}$
is independent on representatives. Indeed, if $(u_{\varepsilon
})_{\varepsilon }$ and $(v_{\varepsilon })_{\varepsilon }$ are
representatives of $\widetilde{u},$ hence by (\ref{continuity}),
\begin{equation*}
\forall \varepsilon \in I,~\forall k\in \mathbb{Z}_{+},~\exists
C_{k}>0,~\left\vert Eu_{\varepsilon }-Ev_{\varepsilon }\right\vert
_{k,\infty ,\mathbb{R}}\leq C_{k}\left\vert u_{\varepsilon }-v_{\varepsilon
}\right\vert _{k,\infty ,\mathbb{J}}.
\end{equation*}%
As $(u_{\varepsilon }-v_{\varepsilon })_{\varepsilon }\in \mathcal{N}_{%
\mathcal{B}}(\mathbb{J})$ then $\forall k\in \mathbb{Z}_{+},~\forall
m>0,~\left\vert Eu_{\varepsilon }-Ev_{\varepsilon }\right\vert _{k,\infty ,%
\mathbb{R}}=O(\varepsilon ^{m}),~\varepsilon \rightarrow 0,$ which shows
that $\left( Eu_{\varepsilon }-Ev_{\varepsilon }\right) _{\varepsilon }\in
\mathcal{N}_{\mathcal{B}}\left( \mathbb{R} \right) .$

We have $\widetilde{E}\widetilde{u}_{\mid \mathbb{J}}=\widetilde{u}$ in $%
\mathcal{G}_{\mathcal{B}}\left( \mathbb{J}\right) .$ Indeed, as $\widetilde{E%
}\widetilde{u}=\left[ \left( Eu_{\varepsilon }\right) _{\varepsilon }\right]
\in \mathcal{G}_{\mathcal{B}}\left( \mathbb{R}\right) $\ and $\widetilde{u}=%
\left[ \left( u_{\varepsilon }\right) _{\varepsilon }\right] \in \mathcal{G}%
_{\mathcal{B}}\left( \mathbb{J}\right) ,$\ therefore $\widetilde{E}%
\widetilde{u}_{\mid \mathbb{J}}-\widetilde{u}:=\left[ \left( Eu_{\varepsilon
\mid \mathbb{J}}\right) _{\varepsilon }\right] -\widetilde{u}=\left[ \left(
u_{\varepsilon }\right) _{\varepsilon }\right] -\widetilde{u}=\widetilde{u}-%
\widetilde{u}=0$ in $\mathcal{G}_{\mathcal{B}}\left( \mathbb{J}\right) .$

If $\widetilde{u}=\left[ \left( u_{\varepsilon }\right) _{\varepsilon }%
\right] \in \mathcal{G}_{+,0}\left( \mathbb{J}\right) \subset \mathcal{G}_{%
\mathcal{B}}\left( \mathbb{J}\right) ,$ then $\widetilde{E}\widetilde{u}=%
\left[ \left( Eu_{\varepsilon }\right) _{\varepsilon }\right] \in \mathcal{G}%
_{\mathcal{B}}\left( \mathbb{R} \right) .$ So $\forall \varepsilon \in
I,Eu_{\varepsilon }\in \mathcal{B}\left( \mathbb{R} \right) .$ The fact that
$\forall \varepsilon \in I,Eu_{\varepsilon }=u_{\varepsilon }$ on $\mathbb{J}%
,$ and $\forall \varepsilon \in I,u_{\varepsilon }\in \mathcal{B}%
_{+,0}\left( \mathbb{J}\right) $ implies $\lim_{x\rightarrow +\infty
}Eu_{\varepsilon }\left( x\right) =0,$ i.e. $\forall \varepsilon \in
I,Eu_{\varepsilon }\in \mathcal{C}_{+,0}.$ By (\cite{BT-AAAD}, Proposition $%
5-\left( 5\right) $), we obtain that $\forall \varepsilon \in
I,Eu_{\varepsilon }\in \mathcal{C}_{+,0}\cap \mathcal{B}\left( \mathbb{R}
\right) =\mathcal{B}_{+,0}\left( \mathbb{R} \right) ,$ it follows that $%
\widetilde{E}\widetilde{u}\in \mathcal{G}_{+,0}\left( \mathbb{R}\right) .$
\end{proof}

\section{The decomposition}

In this section we show that an asymptotically almost automorphic
generalized function is uniquely decomposed as in the classical case.

\begin{theorem}
\label{decomp-AAAGF} Let $\widetilde{u}\in \mathcal{G}_{aaa}\left( \mathbb{R}
\right) $ then there exist $\widetilde{v}\in \mathcal{G}_{aa}\left( \mathbb{R%
} \right) $ and $\widetilde{w}\in \mathcal{G}_{+,0}\left( \mathbb{R} \right)
$ such that $\widetilde{u}=\widetilde{v}+\widetilde{w}$ on $\mathbb{J},$ \
and the decomposition is unique on $\mathbb{J}.$
\end{theorem}

\begin{proof}
Let $\widetilde{u}=[(u_{\varepsilon })_{\varepsilon }]\in \mathcal{G}_{aaa},$
so $\forall \varepsilon \in I,\forall j\in \mathbb{Z}_{+},~u_{\varepsilon
}^{(j)}\in \mathcal{C}_{aaa}.$ Then there exist $v_{\varepsilon ,j}\in
\mathcal{C}_{aa},w_{\varepsilon ,j}\in \mathcal{C}_{+,0},$ such that $%
\forall j\in \mathbb{\mathbb{N} },~u_{\varepsilon }^{(j)}=(v_{\varepsilon
,j}+w_{\varepsilon ,j})\in \mathcal{C}_{aaa}$ on $\mathbb{J},$ and for $j=0,$
$u_{\varepsilon }=v_{\varepsilon }+w_{\varepsilon }$ on $\mathbb{J}.$ By (%
\cite{BT-AAAD}, Proposition $8$), it holds that $\forall j\in \mathbb{%
\mathbb{N} },$ $v_{\varepsilon ,j}=(v_{\varepsilon })^{(j)}$ on $\mathbb{R}$
and $w_{\varepsilon ,j}=(w_{\varepsilon })^{(j)}$ on $\mathbb{J},$ which
gives $v_{\varepsilon }\in \mathcal{B}_{aa}$ and $w_{\varepsilon }\in
\mathcal{B}_{+,0}\left( \mathbb{J}\right) .$ Let's show that $%
(v_{\varepsilon })_{\varepsilon }\in \mathcal{M}_{aa}.$ As $(u_{\varepsilon
})_{\varepsilon }\in \mathcal{M}_{aaa},$ therefore%
\begin{equation}
\forall k\in \mathbb{Z}_{+},~\exists m\in \mathbb{Z}_{+},~\exists
c>0,~\exists \varepsilon _{0}\in I,~\forall \varepsilon <\varepsilon
_{0},~\left\vert u_{\varepsilon }\right\vert _{k,\infty ,\mathbb{R} }\;\leq
c\varepsilon ^{-m}.  \label{u-mod}
\end{equation}%
Due to (\cite{BT-AAAD}, Proposition $3-(5)$), we obtain
\begin{equation}
\forall j\in \mathbb{Z}_{+},\left\Vert v_{\varepsilon }^{(j)}\right\Vert
_{L^{\infty }\left( \mathbb{R} \right) }\leq \left\Vert u_{\varepsilon
}^{(j)}\right\Vert _{L^{\infty }\left( \mathbb{J}\right) },  \label{a1}
\end{equation}%
it follows that
\begin{equation}
\forall k\in \mathbb{Z}_{+},\exists m\in \mathbb{Z}_{+},\exists c>0,\exists
\varepsilon _{0}\in I,\forall \varepsilon <\varepsilon _{0},\left\vert
v_{\varepsilon }\right\vert _{k,\infty ,\mathbb{R} }\leq c\varepsilon ^{-m},
\label{estimae v-mod}
\end{equation}%
this means that $(v_{\varepsilon })_{\varepsilon }\in \mathcal{M}_{aa}.$ If $%
(u_{\varepsilon })_{\varepsilon }\in \mathcal{N}_{aaa}$ then
\begin{equation}
\forall k\in \mathbb{Z}_{+},\forall m\in \mathbb{Z}_{+},\exists c>0,\exists
\varepsilon _{0}\in I,\forall \varepsilon <\varepsilon _{0},\left\vert
u_{\varepsilon }\right\vert _{k,\infty ,\mathbb{R} }\leq c\varepsilon ^{m},
\label{estimae u-null}
\end{equation}%
and from (\ref{a1}) it holds $(v_{\varepsilon })_{\varepsilon }\in \mathcal{N%
}_{aa}.$ Consequently, $\widetilde{v}=[(v_{\varepsilon })_{\varepsilon }]\in
\mathcal{G}_{aa}.$ On the other hand, we have
\begin{equation}
\forall j\in \mathbb{Z}_{+},\left\Vert w_{\varepsilon }^{(j)}\right\Vert
_{L^{\infty }\left( \mathbb{J}\right) }\leq \left\Vert u_{\varepsilon
}^{(j)}\right\Vert _{L^{\infty }\left( \mathbb{J}\right) }+\left\Vert
v_{\varepsilon }^{(j)}\right\Vert _{L^{\infty }\left( \mathbb{J}\right) }.
\label{a2}
\end{equation}%
The estimates (\ref{u-mod}), (\ref{estimae v-mod}) and (\ref{a2}) give
\begin{equation*}
\forall k\in \mathbb{Z}_{+},~\exists m\in \mathbb{Z}_{+},~\exists
c>0,~\exists \varepsilon _{0}\in I,~\forall \varepsilon <\varepsilon
_{0},\left\vert w_{\varepsilon }\right\vert _{k,\infty ,\mathbb{J}}\leq
2c\varepsilon ^{-m},
\end{equation*}%
hence $(w_{\varepsilon })_{\varepsilon }\in \mathcal{M}_{+,0}\left( \mathbb{J%
}\right) .$ If $(u_{\varepsilon })_{\varepsilon }\in \mathcal{N}_{aaa},$
then $(w_{\varepsilon })_{\varepsilon }\in \mathcal{N}_{+,0}\left( \mathbb{J}%
\right) $ follows from (\ref{estimae u-null}) and (\ref{a2}). Thus $%
\widetilde{w}=[(w_{\varepsilon })_{\varepsilon }]\in \mathcal{G}_{+,0}\left(
\mathbb{J}\right) .$ By Theorem \ref{THM2} extending $\widetilde{w}\in
\mathcal{G}_{+,0}\left( \mathbb{J}\right) $ to $\widetilde{E}\widetilde{w}%
\in \mathcal{G}_{+,0}\left( \mathbb{R} \right) $ with $\widetilde{E}%
\widetilde{w}=\widetilde{w}$ on $\mathbb{J}.$\ Finally, $\widetilde{u}=%
\widetilde{v}+\widetilde{w}$ on $\mathbb{J}.$

If $\widetilde{u}\in \mathcal{G}_{aaa}$ has two decompositions, i.e.
\begin{equation*}
\widetilde{u}=\widetilde{v}_{i}+\widetilde{w}_{i}\ \ \text{on}\ \ \mathbb{J}%
,i=1,2,
\end{equation*}%
where $\widetilde{v}_{i}\in \mathcal{G}_{aa}$ and $\widetilde{w}_{i}\in
\mathcal{G}_{+,0}:=\mathcal{G}_{+,0}\left( \mathbb{R} \right) .$ Let $%
(v_{\varepsilon ,i})_{\varepsilon }\in \mathcal{M}_{aa}$ and $%
(w_{\varepsilon ,i})_{\varepsilon }\in \mathcal{M}_{+,0}$ be respectively
representatives of $\widetilde{v}_{i}$ and $\widetilde{w}_{i},i=1,2.$ So $%
\left( v_{\varepsilon ,1}-v_{\varepsilon ,2}\right) _{\varepsilon }+\left(
w_{\varepsilon ,1}-w_{\varepsilon ,2}\right) _{\varepsilon }\in \mathcal{N}_{%
\mathcal{B}}\left( \mathbb{J}\right) ,$ i.e. $\forall k\in \mathbb{Z}%
_{+},\forall m>0,\exists c>0,\exists \varepsilon _{0}\in I,\forall
\varepsilon <\varepsilon _{0},$
\begin{equation}
\left\vert v_{\varepsilon ,1}-v_{\varepsilon ,2}+w_{\varepsilon
,1}-w_{\varepsilon ,2}\right\vert _{k,\infty ,\mathbb{J}}\leq c\varepsilon
^{m}.  \label{unide}
\end{equation}%
Due to (\cite{BT-AAAD}, Proposition $9$) as $\forall \varepsilon \in
I,v_{\varepsilon ,i}\in \mathcal{B}_{aa},i=1,2,$ for any real sequence $%
(s_{m})_{m\in \mathbb{N}},$ such that $s_{m}\rightarrow +\infty $ there
exist $\left( s_{m_{l(\varepsilon )}}\right) _{l}$ a subsequence of $%
(s_{m})_{m\in \mathbb{N}}$ and $g_{\varepsilon ,i},i=1,2,$ such that $%
\forall x\in \mathbb{R},\forall j\in \mathbb{Z}_{+},$
\begin{equation*}
g_{\varepsilon ,i}^{(j)}(x):=\lim\limits_{l\rightarrow +\infty
}v_{\varepsilon ,i}^{(j)}\left( x+s_{m_{l(\varepsilon )}}\right) \ \ \text{%
and }\ \ \lim\limits_{l\rightarrow +\infty }g_{\varepsilon ,i}^{(j)}\left(
x-s_{m_{l(\varepsilon )}}\right) =v_{\varepsilon ,i}^{(j)}(x).
\end{equation*}%
Furthermore, as $\forall \varepsilon \in I,w_{\varepsilon ,i}\in \mathcal{B}%
_{+,0},i=1,2,$
\begin{equation*}
\lim\limits_{l\rightarrow +\infty }w_{\varepsilon ,i}^{(j)}\left(
x+s_{m_{l(\varepsilon )}}\right) =0,\forall x\in \mathbb{R},\forall j\in
\mathbb{Z}_{+}.
\end{equation*}%
By using \ (\ref{unide}) we have $\forall m>0,\exists c>0,\exists
\varepsilon _{0}\in I,\forall \varepsilon <\varepsilon _{0},\forall x\geq
-s_{m_{l(\varepsilon )}},$
\begin{equation*}
\left\vert v_{\varepsilon ,1}^{(j)}\left( x+s_{m_{l(\varepsilon )}}\right)
-v_{\varepsilon ,2}^{(j)}\left( x+s_{m_{l(\varepsilon )}}\right)
+w_{\varepsilon ,1}^{(j)}\left( x+s_{m_{l(\varepsilon )}}\right)
-w_{\varepsilon ,2}^{(j)}\left( x+s_{m_{l(\varepsilon )}}\right) \right\vert
\leq c\varepsilon ^{m},
\end{equation*}%
so when $l\rightarrow +\infty $ we obtain $\forall m>0,\exists c>0,\exists
\varepsilon _{0}\in I,\forall \varepsilon <\varepsilon _{0},\forall x\geq
-s_{m_{l(\varepsilon )}},$
\begin{equation*}
\left\vert g_{\varepsilon ,1}^{(j)}(x)-g_{\varepsilon
,2}^{(j)}(x)\right\vert \leq c\varepsilon ^{m},
\end{equation*}%
by taking the translate $-s_{m_{l(\varepsilon )}}$ and let $l\rightarrow
+\infty $ we get $\forall m>0,\exists c>0,\exists \varepsilon _{0}\in
I,\forall \varepsilon <\varepsilon _{0},\forall x\geq 0,$
\begin{equation*}
\left\vert v_{\varepsilon ,1}^{(j)}(x)-v_{\varepsilon
,2}^{(j)}(x)\right\vert \leq c\varepsilon ^{m}.
\end{equation*}%
By (\cite{BT-AAAD}, Proposition $3-(5)$), \ it follows
\begin{equation}
\forall k\in \mathbb{Z}_{+},\forall m>0,\exists c>0,\exists \varepsilon
_{0}\in I,\forall \varepsilon <\varepsilon _{0},\left\vert v_{\varepsilon
,1}-v_{\varepsilon ,2}\right\vert _{k,\infty ,\mathbb{R}}\leq c\varepsilon
^{m},  \label{t-pr}
\end{equation}%
which shows that $(v_{\varepsilon ,1}-v_{\varepsilon ,2})_{\varepsilon }\in
\mathcal{N}_{\mathcal{B}}\left( \mathbb{R}\right) ,$ so $\widetilde{v}_{1}=%
\widetilde{v}_{2}$ in $\mathcal{G}_{\mathcal{B}}\left( \mathbb{R}\right) .$
From (\ref{unide}) and (\ref{t-pr}) it holds that $\left( w_{\varepsilon
,1}-w_{\varepsilon ,2}\right) _{\varepsilon }\in \mathcal{N}_{\mathcal{B}%
}\left( \mathbb{J}\right) ,$ i.e. $\widetilde{w}_{1}=\widetilde{w}_{2}$ on $%
\mathbb{J}.$
\end{proof}

\begin{notation}
Let $\widetilde{u}\in \mathcal{G}_{aaa}$ and $\widetilde{u}=\widetilde{v}+%
\widetilde{w}$ on $\mathbb{J},$ where $\widetilde{v}\in \mathcal{G}_{aa}$
and $\widetilde{w}\in \mathcal{G}_{+,0},$ then $\widetilde{v}$ and $%
\widetilde{w}$ are called respectively the principal term and the corrective
term of $\widetilde{u}$ and we denote them respectively $\widetilde{u}_{aa}$
and $\widetilde{u}_{cor}.$ Also $\widetilde{u}=\left( \widetilde{u}_{aa}+%
\widetilde{u}_{cor}\right) \in \mathcal{G}_{aaa}$ means that $\widetilde{u}%
_{aa}\in \mathcal{G}_{aa},$ $\widetilde{u}_{cor}\in \mathcal{G}_{+,0}$ and $%
\widetilde{u}=\widetilde{u}_{aa}+\widetilde{u}_{cor}$ on $\mathbb{J}.$
\end{notation}

\section{Non linear operations}

The algebra of tempered generalized functions on $\mathbb{\mathbb{C}}$
denoted by $\mathcal{G}_{\tau }(\mathbb{\mathbb{C}}),$ see \cite{G} and \cite%
{NPS}\ for more details, is the quotient algebra
\begin{equation*}
\mathcal{G}_{\tau }(\mathbb{\ \mathbb{C}}):=\frac{\mathcal{M}_{\tau }(%
\mathbb{\ \mathbb{C}})}{\mathcal{N}_{\tau }(\mathbb{\mathbb{C}})},
\end{equation*}%
where
\begin{equation*}
\mathcal{M}_{\tau }(\mathbb{\ \mathbb{C}}):=\left\{
\begin{array}{c}
\left( f_{\varepsilon }\right) _{\varepsilon }\in \left( \mathcal{E}\left(
\mathbb{R}^{2}\right) \right) ^{I}:\forall j\in \mathbb{Z}_{+}^{2},\exists
m\in \mathbb{Z}_{+}, \\
\sup_{x\in \mathbb{R}^{2}}\left( 1+\left\vert x\right\vert \right)
^{-m}\left\vert f_{\varepsilon }^{(j)}(x)\right\vert =O\left( \varepsilon
^{-m}\right) ,\varepsilon \rightarrow 0%
\end{array}%
\right\} .
\end{equation*}%
\begin{equation*}
\mathcal{N}_{\tau }(\mathbb{\ \mathbb{C}}):=\left\{
\begin{array}{c}
\left( f_{\varepsilon }\right) _{\varepsilon }\in \left( \mathcal{E}\left(
\mathbb{R}^{2}\right) \right) ^{I}:\forall j\in \mathbb{Z}_{+}^{2},\exists
n\in \mathbb{Z}_{+},\forall m\in \mathbb{Z}_{+}, \\
\sup_{x\in \mathbb{R}^{2}}\left( 1+\left\vert x\right\vert \right)
^{-n}\left\vert f_{\varepsilon }^{(j)}(x)\right\vert =O\left( \varepsilon
^{m}\right) ,\varepsilon \rightarrow 0%
\end{array}%
\right\} .
\end{equation*}

\begin{example}
Any polynomial function is a tempered generalized function.
\end{example}

\begin{theorem}
\label{composition} Let $\widetilde{u}=\left[ \left( u_{\varepsilon }\right)
_{\varepsilon }\right] \in \mathcal{G}_{aaa}$ and $\widetilde{F}=\left[
\left( f_{\varepsilon }\right) _{\varepsilon }\right] \in \mathcal{G}_{\tau
}(\mathbb{\mathbb{C} }),$ then
\begin{equation*}
\widetilde{F}\circ \widetilde{u}:=[(f_{\varepsilon }\circ u_{\varepsilon
})_{\varepsilon }]
\end{equation*}%
is a well-defined element of $\mathcal{G}_{aaa}.$ The principal term and the
corrective term of $\widetilde{F}\circ \widetilde{u}$ are respectively $%
\widetilde{F}\left( \widetilde{u}_{aa}\right) $ and $\widetilde{F}\left(
\widetilde{u}_{aa}+\widetilde{u}_{cor}\right) -\widetilde{F}\left(
\widetilde{u}_{aa}\right) ,$ where $\widetilde{u}=\widetilde{u}_{aa}+%
\widetilde{u}_{cor}$ on $\mathbb{J}.$
\end{theorem}

\begin{proof}
Let $(u_{\varepsilon })_{\varepsilon }\in \mathcal{M}_{aaa}$ and $%
(f_{\varepsilon })_{\varepsilon }\in \mathcal{M}_{\tau }\left( \mathbb{%
\mathbb{C} }\right) ,$ by the classical Fa\`{a}~di Bruno formula, we have $%
\forall j\in \mathbb{Z}_{+},$
\begin{equation}
\dfrac{(f_{\varepsilon }\circ u_{\varepsilon })^{(j)}(x)}{j!}=\sum\limits
_{\substack{ l_{1}+2l_{2}+\cdots +jl_{j}=j  \\ r=l_{1}+\cdots +l_{j}}}\dfrac{%
f_{\varepsilon }^{(r)}(u_{\varepsilon }(x))}{l_{1}!\cdots l_{j}!}%
\prod\limits_{i=1}^{j}\left( \dfrac{u_{\varepsilon }^{(i)}(x)}{i!}\right)
^{l_{i}}.  \label{F-B}
\end{equation}%
As $\forall \varepsilon \in I,~\forall j\in \mathbb{Z}_{+},~u_{\varepsilon
}^{(j)}\in \mathcal{C}_{aaa}$ and $f_{\varepsilon }\in \mathcal{E}\left(
\mathbb{R} \right) ,$ it follows by (\cite{BT-AAAD}, Proposition $3-(4)$),
that $f_{\varepsilon }^{(r)}(u_{\varepsilon })\in \mathcal{C}_{aaa},$ and
since $\mathcal{C}_{aaa}$ is an algebra, then $\forall \varepsilon \in
I,~f_{\varepsilon }\circ u_{\varepsilon }\in \mathcal{B}_{aaa}.$ As $%
(u_{\varepsilon })_{\varepsilon }\in \mathcal{M}_{aaa},$ then
\begin{equation*}
\forall k\in \mathbb{Z}_{+},\exists n_{k}\in \mathbb{Z}_{+},\exists
c_{k}>0,\exists \varepsilon _{k}\in I,\forall \varepsilon <\varepsilon
_{k},\left\vert u_{\varepsilon }\right\vert _{k,\infty ,\mathbb{R} }\leq
c_{k}\varepsilon ^{-n_{k}}.
\end{equation*}%
The fact that $(f_{\varepsilon })_{\varepsilon }\in \mathcal{M}_{\tau
}\left( \mathbb{\mathbb{C} }\right) $ gives
\begin{equation*}
\forall j\in \mathbb{Z}_{+},\exists N_{j}\in \mathbb{Z}_{+},\exists
C_{j}>0,\exists \varepsilon _{j}^{^{\prime }}\in I,\forall \varepsilon
<\varepsilon _{j}^{^{\prime }},\left\Vert f_{\varepsilon
}^{(j)}(u_{\varepsilon })\right\Vert _{L^{\infty }\left( \mathbb{R} \right)
}\leq C_{j}\varepsilon ^{-N_{j}}\left\Vert 1+u_{\varepsilon }\right\Vert
_{L^{\infty }\left( \mathbb{R} \right) }^{N_{j}}.
\end{equation*}%
Consequently, by (\ref{F-B})\ we obtain
\begin{equation*}
\dfrac{\left\Vert (f_{\varepsilon }\circ u_{\varepsilon })^{(j)}\right\Vert
_{L^{\infty }\left( \mathbb{R} \right) }}{j!}\leq \sum\limits_{\substack{ %
l_{1}+2l_{2}+\cdots +jl_{j}=j  \\ r=l_{1}+\cdots +l_{j}}}\dfrac{%
C_{r}\varepsilon ^{-N_{r}}\left\Vert 1+u_{\varepsilon }\right\Vert
_{L^{\infty }\left( \mathbb{R} \right) }^{N_{r}}}{l_{1}!\cdots l_{j}!}%
\prod\limits_{i=1}^{j}\left( \dfrac{\left\Vert u_{\varepsilon
}^{(i)}\right\Vert _{L^{\infty }\left( \mathbb{R} \right) }}{i!}\right)
^{l_{i}},
\end{equation*}%
hence there exists $c>0,$
\begin{eqnarray*}
\dfrac{\left\Vert (f_{\varepsilon }\circ u_{\varepsilon })^{(j)}\right\Vert
_{L^{\infty }\left( \mathbb{R} \right) }}{j!} &\leq &\sum\limits_{\substack{ %
l_{1}+2l_{2}+\cdots +jl_{j}=j  \\ r=l_{1}+\cdots +l_{j}}}\dfrac{c\varepsilon
^{-N_{r}(1+n_{0})}}{l_{1}!\cdots l_{j}!}\prod_{i=1}^{j}\left( \dfrac{%
c_{i}\varepsilon ^{-n_{i}}}{i!}\right) ^{l_{i}}, \\
&\leq &\sum\limits_{\substack{ l_{1}+2l_{2}+\cdots +jl_{j}=j  \\ %
r=l_{1}+\cdots +l_{j}}}\dfrac{c\varepsilon
^{-(N_{r}(1+n_{0})+\sum\limits_{i=1}^{j}n_{i}l_{i})}}{l_{1}!\cdots l_{j}!}%
\prod\limits_{i=1}^{j}\left( \dfrac{c_{i}}{i!}\right) ^{l_{i}}, \\
&\leq &C^{^{\prime }}\varepsilon ^{-m},
\end{eqnarray*}%
where
\begin{equation*}
m=\max\limits_{\substack{ _{\substack{ l_{1}+2l_{2}+\cdots +jl_{j}=j  \\ %
r=l_{1}+\cdots +l_{j}}}  \\ 1\leq r\leq j}}\left\{
N_{r}(1+n_{0})+\sum\limits_{i=1}^{j}n_{i}l_{i})\right\} ,C^{^{\prime
}}=\sum\limits_{\substack{ l_{1}+2l_{2}+\cdots +jl_{j}=j  \\ r=l_{1}+\cdots
+l_{j}}}\dfrac{c}{l_{1}!\cdots l_{j}!}\prod\limits_{i=1}^{j}\left( \dfrac{%
c_{i}}{i!}\right) ^{l_{i}}.\newline
\end{equation*}%
Finally, with $C=C^{^{\prime }}\sum\limits_{j\leq k}j!,$ it holds
\begin{equation*}
\forall k\in \mathbb{Z}_{+},\exists m\in \mathbb{Z}_{+},\exists C>0,\exists
\varepsilon ^{"}=\inf_{1\leq i\leq j\leq k}\left( \varepsilon
_{i},\varepsilon _{j}^{^{\prime }}\right) ,\forall \varepsilon <\varepsilon
^{"},\left\vert f_{\varepsilon }\circ u_{\varepsilon }\right\vert _{k,\infty
,\mathbb{R} }\leq C\varepsilon ^{-m},
\end{equation*}%
which means that $(f_{\varepsilon }\circ u_{\varepsilon })_{\varepsilon }\in
\mathcal{M}_{aaa}.$ This composition does not depend on the representatives.
Indeed, suppose that $(v_{\varepsilon })_{\varepsilon }\in \mathcal{M}_{aaa}$
and $(g_{\varepsilon })_{\varepsilon }\in \mathcal{M}_{\tau }\left( \mathbb{%
\mathbb{C} }\right) $ are others representatives of $\widetilde{u}$ and $%
\widetilde{F}$ respectively. Set $(n_{\varepsilon })_{\varepsilon
}:=((v_{\varepsilon })_{\varepsilon }-(u_{\varepsilon })_{\varepsilon })\in
\mathcal{N}_{aaa}$ and $(m_{\varepsilon })_{\varepsilon }:=\left(
(f_{\varepsilon })_{\varepsilon }-(g_{\varepsilon })_{\varepsilon }\right)
\in \mathcal{N}_{\tau }\left( \mathbb{\mathbb{C} }\right) .$ To show that $%
(f_{\varepsilon }\circ u_{\varepsilon }-g_{\varepsilon }\circ v_{\varepsilon
})_{\varepsilon }\in \mathcal{N}_{aaa},$ since $(f_{\varepsilon }\circ
u_{\varepsilon }-g_{\varepsilon }\circ v_{\varepsilon })_{\varepsilon }\in
\mathcal{M}_{aaa},$\ according to Proposition \ref{M-N}$-\left( 1\right) ,$
it is enough to prove that $(f_{\varepsilon }\circ u_{\varepsilon
}-g_{\varepsilon }\circ v_{\varepsilon })_{\varepsilon }$ satisfies (\ref%
{equ8}). Indeed, we have
\begin{eqnarray*}
\left\vert f_{\varepsilon }\circ u_{\varepsilon }-g_{\varepsilon }\circ
v_{\varepsilon }\right\vert _{0,\infty ,\mathbb{R} } &\leq &\left\vert
f_{\varepsilon }\circ u_{\varepsilon }-f_{\varepsilon }\circ v_{\varepsilon
}\right\vert _{0,\infty ,\mathbb{R} }+\left\vert f_{\varepsilon }\circ
v_{\varepsilon }-g_{\varepsilon }\circ v_{\varepsilon }\right\vert
_{0,\infty ,\mathbb{R} }, \\
&=&\left\vert f_{\varepsilon }(u_{\varepsilon })-f_{\varepsilon
}(u_{\varepsilon }+n_{\varepsilon })\right\vert _{0,\infty ,\mathbb{R}
}+\left\vert m_{\varepsilon }(v_{\varepsilon })\right\vert _{0,\infty ,%
\mathbb{R} }, \\
&\leq &\left\vert n_{\varepsilon }\right\vert _{0,\infty ,\mathbb{R}
}\left\vert f_{\varepsilon }^{^{\prime }}(u_{\varepsilon })\right\vert
_{0,\infty ,\mathbb{R} }+\left\vert m_{\varepsilon }(v_{\varepsilon
})\right\vert _{0,\infty ,\mathbb{R} }.
\end{eqnarray*}%
It is clear that $\forall k\in \mathbb{Z}_{+},~\left\vert n_{\varepsilon
}\right\vert _{0,\infty ,\mathbb{R} }\left\vert f_{\varepsilon }^{^{\prime
}}(u_{\varepsilon })\right\vert _{0,\infty ,\mathbb{R} }=O(\varepsilon
^{k}),~\varepsilon \rightarrow 0,$ and also $\forall l\in \mathbb{Z}%
_{+},~\left\vert m_{\varepsilon }(v_{\varepsilon })\right\vert _{0,\infty ,%
\mathbb{R} }=O(\varepsilon ^{l}),~\varepsilon \rightarrow 0.$ Therefore, $%
\forall q\in \mathbb{Z}_{+},~\left\vert f_{\varepsilon }\circ u_{\varepsilon
}-g_{\varepsilon }\circ v_{\varepsilon }\right\vert _{0,\infty ,\mathbb{R}
}=O(\varepsilon ^{q}),~\varepsilon \rightarrow 0.$

Let $\widetilde{u}=\left( \widetilde{u}_{aa}+\widetilde{u}_{cor}\right) \in
\mathcal{G}_{aaa}.$ As $\widetilde{F}\circ \widetilde{u}=\widetilde{F}\left(
\widetilde{u}_{aa}\right) +\left( \widetilde{F}\left( \widetilde{u}\right) -%
\widetilde{F}\left( \widetilde{u}_{aa}\right) \right) ,$ then $\widetilde{F}%
\circ \widetilde{u}=\widetilde{F}\left( \widetilde{u}_{aa}\right) +\left(
\widetilde{F}\left( \widetilde{u}_{aa}+\widetilde{u}_{cor}\right) -%
\widetilde{F}\left( \widetilde{u}_{aa}\right) \right) $ on $\mathbb{J}.$ In
view of (\cite{BKT-AAGF}, Proposition $9$) we obtain $\widetilde{F}\left(
\widetilde{u}_{aa}\right) \in \mathcal{G}_{aa}.$ \ It remains to prove that $%
\widetilde{F}\left( \widetilde{u}_{aa}+\widetilde{u}_{cor}\right) -%
\widetilde{F}\left( \widetilde{u}_{aa}\right) \in \mathcal{G}_{_{+,0}}.$
Since $\mathcal{G}_{aaa}$ and $\mathcal{G}_{aa}$ are contained in $\mathcal{G%
}_{\mathcal{B}}$ then $\widetilde{F}\left( \widetilde{u}_{aa}+\widetilde{u}%
_{cor}\right) -\widetilde{F}\left( \widetilde{u}_{aa}\right) \in \mathcal{G}%
_{\mathcal{B}}.$ It suffices to show that $\forall \varepsilon \in
I,f_{\varepsilon }\left( u_{aa,\varepsilon }+u_{cor,\varepsilon }\right)
-f_{\varepsilon }\left( u_{aa,\varepsilon }\right) \in \mathcal{B}_{+,0},$
where $(f_{\varepsilon })_{\varepsilon },(u_{aa,\varepsilon })_{\varepsilon
} $ and $(u_{cor},_{\varepsilon })_{\varepsilon }$ are respective
representatives of $\widetilde{F},\widetilde{u}_{aa}$ and $\widetilde{u}%
_{cor}.$ The classical result on composition of asymptotically almost
automorphic function with continuous function shows that the corrective term
of $\ f_{\varepsilon }\left( u_{aa,\varepsilon }+u_{cor,\varepsilon }\right)
$ is $f_{\varepsilon }\left( u_{aa,\varepsilon }+u_{cor,\varepsilon }\right)
-f_{\varepsilon }\left( u_{aa,\varepsilon }\right) $ and the fact that $%
\widetilde{F}\left( \widetilde{u}_{aa}+\widetilde{u}_{cor}\right) -%
\widetilde{F}\left( \widetilde{u}_{aa}\right) \in \mathcal{G}_{\mathcal{B}}$
gives $\forall \varepsilon \in I,$ $\left( f_{\varepsilon }\left(
u_{aa,\varepsilon }+u_{cor,\varepsilon }\right) -f_{\varepsilon }\left(
u_{aa,\varepsilon }\right) \right) \in \mathcal{B}.$ By (\cite{BT-AAAD},
Proposition $5-(5)$), we have $\forall \varepsilon \in I,f_{\varepsilon
}\left( u_{aa,\varepsilon }+u_{cor,\varepsilon }\right) -f_{\varepsilon
}\left( u_{aa,\varepsilon }\right) \in \mathcal{C}_{+,0}\cap \mathcal{B=B}%
_{+,0}.$
\end{proof}

\section{Linear neutral difference differential systems}

We consider linear neutral difference differential systems for the unknown
vector function $\widetilde{u}=\left( \widetilde{u}_{1},\ldots ,\widetilde{u}%
_{n}\right) ,$%
\begin{equation}
L_{\omega }\widetilde{u}:=\sum\limits_{i=0}^{p}\sum\limits_{j=0}^{q}%
\widetilde{A}_{ij}\left( \tau _{\omega _{j}}\widetilde{u}\right) ^{(i)}+%
\widetilde{K}\ast \widetilde{u}=\widetilde{f},  \label{equation}
\end{equation}%
where $\widetilde{f}=\left( \widetilde{f}_{1},\ldots ,\widetilde{f}%
_{n}\right) \in \left( \mathcal{G}_{aaa}\right) ^{n},\omega =\left( \omega
_{j}\right) _{0\leq j\leq q}\subset \mathbb{R} _{+}^{q}$ and for $i\leq
p,j\leq q,\widetilde{A}_{ij}=\left( \widetilde{A}_{ij}^{rl}\right) _{1\leq
r,l\leq n}$ and $\widetilde{K}=\left( \widetilde{K}^{rl}\right) _{1\leq
r,l\leq n}$ are square matrices of almost automorphic generalized functions
and $L^{1}-$generalized functions respectively.

\begin{lemma}
If $\widetilde{u}\in \left( \mathcal{G}_{aaa}\right) ^{n}$ then $L_{\omega }
\widetilde{u}\in \left( \mathcal{G}_{aaa}\right) ^{n}.$
\end{lemma}

\begin{proof}
If $\widetilde{u}\in \left( \mathcal{G}_{aaa}\right) ^{n}$ then due to
results of Proposition \ref{prop3.2} we obtain that $\forall i\leq p,j\leq
q,\forall \omega =\left( \omega _{j}\right) _{0\leq j\leq q}\subset \mathbb{R%
} _{+}^{q},$
\begin{equation*}
\left( \widetilde{A}_{ij}\left( \tau _{\omega _{j}}\widetilde{u}\right)
^{\left( i\right) }\right) \in \left( \mathcal{G}_{aaa}\right) ^{n},
\end{equation*}%
and also
\begin{equation*}
\widetilde{K}\ast \widetilde{u}\in \left( \mathcal{G}_{aaa}\right) ^{n},
\end{equation*}%
so
\begin{equation*}
\left( \sum\limits_{i=0}^{p}\sum\limits_{j=0}^{q}\widetilde{A}_{ij}\left(
\tau _{\omega _{j}}\widetilde{u}\right) ^{(i)}+\widetilde{K}\ast \widetilde{u%
}\right) \in \left( \mathcal{G}_{aaa}\right) ^{n},
\end{equation*}%
i.e. $L_{\omega }\widetilde{u}\in \left( \mathcal{G}_{aaa}\right) ^{n}.$
\end{proof}

\begin{definition}
A generalized function $\widetilde{u}\in \left( \mathcal{G}_{\mathcal{B}%
}\right) ^{n}$ is said a generalized solution of (\ref{equation}) on $%
\mathbb{J}$ if it satisfies
\begin{equation*}
\left(
\sum\limits_{l=1}^{n}\sum\limits_{i=1}^{p}\sum\limits_{j=1}^{q}A_{ij,%
\varepsilon }^{rl}\left( \tau _{\omega _{j}}u_{l,\varepsilon }\right)
^{\left( i\right) }+K_{\varepsilon }^{rl}\ast u_{l,\varepsilon
}-f_{r,\varepsilon }\right) _{\varepsilon }\in \mathcal{N}_{\mathcal{B}%
}\left( \mathbb{J}\right) ,r=1,...,n,
\end{equation*}%
where $\left( u_{l,\varepsilon }\right) _{\varepsilon },\left(
A_{ij,\varepsilon }^{rl}\right) _{\varepsilon },\left( K_{\varepsilon
}^{rl}\right) _{\varepsilon }$ and $\left( f_{r,\varepsilon }\right)
_{\varepsilon },r,l=1,...,n,$ are respective representatives of $\widetilde{u%
}_{r},\widetilde{A}_{ij}^{rl},\widetilde{K}^{rl}$ and $\widetilde{f}_{r}.$
\end{definition}

We give the main result of this section.

\begin{theorem}
\label{thm-eqt} Let $\widetilde{f}=\left( \widetilde{f}_{aa}+\widetilde{f}%
_{cor}\right) \in \left( \mathcal{G}_{aaa}\right) ^{n},$ the equation (\ref%
{equation}) admits a generalized solution $\widetilde{u}\in \left( \mathcal{G%
}_{aaa}\right) ^{n}$ on $\mathbb{J}$ if and only if there exist $\widetilde{v%
}\in \left( \mathcal{G}_{aa}\right) ^{n}$ and $\widetilde{w}\in \left(
\mathcal{G}_{+,0}\right) ^{n}$ such that
\begin{equation}
L_{\omega }\widetilde{v}=\widetilde{f}_{aa}\ \ \text{on}\ \ \mathbb{R},
\label{aa-eqt}
\end{equation}%
and
\begin{equation}
L_{\omega }\widetilde{w}=\widetilde{f}_{cor}\ \ \text{on}\ \ \mathbb{J}.
\label{cor-eqt}
\end{equation}
\end{theorem}

\begin{proof}
If $\widetilde{u}=(\widetilde{u}_{aa}+\widetilde{u}_{cor})\in \left(
\mathcal{G}_{aaa}\right) ^{n}$ is a generalized solution of (\ref{equation}%
), then the equations (\ref{equation}) explicitly are written as follows

\begin{equation*}
\left\{
\begin{array}{l}
\sum\limits_{l=1}^{n}\left( \sum\limits_{i=0}^{p}\sum\limits_{j=0}^{q}%
\widetilde{A}_{ij}^{1l}\left( \tau _{\omega _{j}}\widetilde{u}_{l}\right)
^{(i)}+\widetilde{K}^{1l}\ast \widetilde{u}_{l}\right) =\widetilde{f}_{1}.
\\
\vdots \\
\sum\limits_{l=1}^{n}\left( \sum\limits_{i=0}^{p}\sum\limits_{j=0}^{q}%
\widetilde{A}_{ij}^{nl}\left( \tau _{\omega _{j}}\widetilde{u}_{l}\right)
^{(i)}+\widetilde{K}^{nl}\ast \widetilde{u}_{l}\right) =\widetilde{f}_{n}.%
\end{array}%
\right.
\end{equation*}%
Let $(u_{aa,l,\varepsilon })_{\varepsilon },(u_{cor,l,\varepsilon
})_{\varepsilon },\left( A_{ij,\varepsilon }^{rl}\right) _{\varepsilon
},\left( K_{\varepsilon }^{rl}\right) _{\varepsilon },$ $(f_{aa,l,%
\varepsilon })_{\varepsilon }$ and $(f_{cor,l,\varepsilon })_{\varepsilon }$
be respective representatives of $\widetilde{u}_{aa,l},$ $\ \widetilde{u}%
_{cor,l},\widetilde{A}_{ij}^{rl},\widetilde{K}^{rl},\widetilde{f}_{aa,l}$
and $\widetilde{f}_{cor,l},1\leq l\leq n.$ For every $r=1,\ldots ,n$ and $%
\varepsilon \in I,$ denoting
\begin{equation*}
S_{aa,r,\varepsilon }:=\sum\limits_{l=1}^{n}\left(
\sum\limits_{i=1}^{p}\sum\limits_{j=1}^{q}A_{ij,\varepsilon }^{rl}\left(
\tau _{\omega _{j}}u_{aa,l,\varepsilon }\right) ^{\left( i\right)
}+K_{\varepsilon }^{rl}\ast u_{aa,l,\varepsilon }\right) ,\text{ }
\end{equation*}%
\begin{equation*}
S_{cor,r,\varepsilon }:=\sum\limits_{l=1}^{n}\left(
\sum\limits_{i=1}^{p}\sum\limits_{j=1}^{q}A_{ij,\varepsilon }^{rl}\left(
\tau _{\omega _{j}}u_{cor,l,\varepsilon }\right) ^{\left( i\right)
}+K_{\varepsilon }^{rl}\ast u_{cor,l,\varepsilon }\right) ,
\end{equation*}%
we have $\forall \varepsilon \in I,S_{aa,r,\varepsilon }\in \mathcal{B}_{aa}$
and $S_{cor,r,\varepsilon }\in \mathcal{B}_{+,0},$\ and also $\left[ \left(
S_{aa,r,\varepsilon }\right) _{\varepsilon }\right] \in \mathcal{G}_{aa}$
and $\left[ \left( S_{cor,r,\varepsilon }\right) _{\varepsilon }\right] \in
\mathcal{G}_{+,0}.$ Since $\widetilde{u}$ is a generalized solution of (\ref%
{equation}) on $\mathbb{J},$ it means that $\forall k\in \mathbb{Z}%
_{+},\forall m>0,\exists c_{k}>0,\exists \varepsilon _{k}\in I,\forall
\varepsilon <\varepsilon _{k},\forall x\in \mathbb{J},$
\begin{equation}
\sum\limits_{j\leq k}\left\vert S_{aa,r,\varepsilon }^{(j)}\left( x\right)
-f_{aa,r,\varepsilon }^{(j)}\left( x\right) +S_{cor,r,\varepsilon
}^{(j)}\left( x\right) -f_{cor,r,\varepsilon }^{(j)}\left( x\right)
\right\vert \leq c_{k}\varepsilon ^{m},r=1,\ldots ,n.  \label{1-n}
\end{equation}%
For any real sequence $(s_{m})_{m\in \mathbb{N}},$ such that $%
s_{m}\rightarrow +\infty $ there exist $\left( s_{m_{p(\varepsilon
)}}\right) _{p}$ a subsequence of $(s_{m})_{m\in \mathbb{N}},$ such that
taking the translate at $s_{m_{p(\varepsilon )}}-s_{m_{q(\varepsilon )}}$ in
the estimate (\ref{1-n}) and let $p,q\rightarrow +\infty ,$ we obtain due to
(\cite{BT-AAAD}, Proposition $9$) that $\forall k\in \mathbb{Z}_{+},\forall
m>0,\exists c_{k}>0,\exists \varepsilon _{k}\in I,\forall \varepsilon
<\varepsilon _{k},\forall x\in \mathbb{J},$%
\begin{equation*}
\sum\limits_{j\leq k}\left\vert S_{aa,r,\varepsilon }^{(j)}\left( x\right)
-f_{aa,r,\varepsilon }^{(j)}\left( x\right) \right\vert \leq
c_{k}\varepsilon ^{m},\forall r=1,\ldots ,n,
\end{equation*}%
consequently, by Proposition $3-(5)$ of \cite{BT-AAAD}, \ it holds $\forall
k\in \mathbb{Z}_{+},\forall m>0,\exists c_{k}>0,\exists \varepsilon _{k}\in
I,\forall \varepsilon <\varepsilon _{k},$%
\begin{equation*}
\left\vert S_{aa,r,\varepsilon }-f_{aa,r,\varepsilon }\right\vert _{k,\infty
,\mathbb{R} }\leq c_{k}\varepsilon ^{m},\forall r=1,\ldots ,n,
\end{equation*}%
i.e.
\begin{equation}
\left( S_{aa,r,\varepsilon }-f_{aa,r,\varepsilon }\right) _{\varepsilon }\in
\mathcal{N}_{\mathcal{B}}\left( \mathbb{R} \right) ,\forall r=1,\ldots ,n,
\label{2-n}
\end{equation}%
which means that $\widetilde{u}_{aa}=\left( \widetilde{u}_{aa,1},\ldots ,%
\widetilde{u}_{aa,n}\right) $ is a generalized solution of (\ref{aa-eqt}) on
$\mathbb{R} .$ By (\ref{1-n}) and (\ref{2-n}), we deduce
\begin{equation*}
\left( S_{cor,r,\varepsilon }-f_{cor,r,\varepsilon }\right) _{\varepsilon
}\in \mathcal{N}_{\mathcal{B}}\left( \mathbb{J}\right) ,\forall r=1,\ldots
,n,
\end{equation*}%
i.e. $\widetilde{u}_{cor}=\left( \widetilde{u}_{cor,1},\ldots ,\widetilde{u}%
_{cor,n}\right) $ is a generalized solution of (\ref{cor-eqt}) on $\mathbb{J}%
.$

Conversely if there exist $\widetilde{v}\in \left( \mathcal{G}_{aa}\right)
^{n}$ and $\widetilde{w}\in \left( \mathcal{G}_{+,0}\right) ^{n}$ such that (%
\ref{aa-eqt}) and (\ref{cor-eqt}) hold, then we have $\widetilde{u}:=(%
\widetilde{v}+\widetilde{w})\in \left( \mathcal{G}_{aaa}\right) ^{n}$ is a
generalized solution of (\ref{equation}) on $\mathbb{J}.$
\end{proof}

\begin{remark}
Theorem \ref{thm-eqt} generalizes Theorem $6$\ of \cite{BK-AAPGFDE} and
Theorem $3$ of \cite{BT-AAAD}.
\end{remark}

As a particular case we consider linear systems of ordinary differential
equations
\begin{equation}
\overset{\cdot }{\widetilde{u}}+\widetilde{A}\widetilde{u}=\widetilde{f},
\label{LSE}
\end{equation}%
where $\widetilde{A}$ is a square matrix of almost automorphic generalized
functions.

\begin{corollary}
\label{cor-app} Let $\widetilde{f}=\left( \widetilde{f}_{aa}+\widetilde{f}%
_{cor}\right) \in \left( \mathcal{G}_{aaa}\right) ^{n},$ the system (\ref%
{LSE}) admits a generalized solution $\widetilde{u}\in \left( \mathcal{G}%
_{aaa}\right) ^{n}$ on $\mathbb{\ J}$ if and only if there exist $\widetilde{%
v}\in \left( \mathcal{G}_{aa}\right) ^{n}$ and $\widetilde{w}\in \left(
\mathcal{G}_{+,0}\right) ^{n}$ such that
\begin{equation*}
\overset{\cdot }{\widetilde{v}}+\widetilde{A}\widetilde{v}=\widetilde{f}_{aa}%
\text{ on }\mathbb{R},
\end{equation*}%
and
\begin{equation*}
\overset{\cdot }{\widetilde{w}}+\widetilde{A}\widetilde{w}=\widetilde{f}%
_{cor}\text{ on }\mathbb{J}.
\end{equation*}

Let $\widetilde{u}=\left[ \left( u_{\varepsilon }\right) _{\varepsilon }%
\right] \in \mathcal{G}_{\mathcal{B}},x_{0}\in \mathbb{R} $ and define the
primitive of $\widetilde{u}$ by $\widetilde{U}=\left[ \left( U_{\varepsilon
}\right) _{\varepsilon }\right] ,$ where%
\begin{equation*}
U_{\varepsilon }(x):=\int\limits_{x_{0}}^{x}u_{\varepsilon
}(t)dt,\varepsilon \in I.
\end{equation*}
\end{corollary}

\begin{corollary}
A generalized function $\widetilde{U}=\left( \widetilde{U}_{aa}+\widetilde{U}%
_{cor}\right) \in \mathcal{G}_{\mathcal{B}}$ is a primitive of $\widetilde{u}%
=(\widetilde{u}_{aa}+\widetilde{u}_{cor})\in \mathcal{G}_{aaa}$ on $\mathbb{J%
}$ if and only if
\begin{equation*}
\widetilde{U}_{aa}\ \ \text{is a primtive of}\ \ \widetilde{u}_{aa}\ \ \text{
on}\ \ \mathbb{R},
\end{equation*}%
and
\begin{equation*}
\widetilde{U}_{cor}\ \ \text{is a primitive of}\ \ \widetilde{u}_{cor}%
\newline
\ \text{on}\ \ \mathbb{J}.
\end{equation*}
\end{corollary}

\end{document}